\documentclass{amsart}
\usepackage{ifthen}
\usepackage{verbatim}
\usepackage{amssymb,amsmath}
\usepackage{enumerate}
\usepackage[usenames,dvipsnames,svgnames,table]{xcolor}
\usepackage[colorlinks=true]{hyperref}

\usepackage{etoolbox}
\AtBeginEnvironment{proof}{\vspace{-14pt}}

\numberwithin{equation}{section}

\def\phi{\varphi}
\def\R{\mathbb R}
\def\N{\mathbb N}
\newcommand{\ch}{\mathcal{H}}
\newcommand{\cm}{\mathcal{M}}

\newcommand{\bfM}{\mathbf{M}}
\newcommand{\bfH}{\mathbf{H}}

\newcommand{\eps}{\varepsilon}
\newcommand{\ra}{\rightarrow}
\newcommand{\be}{\begin{equation}}
\newcommand{\ee}{\end{equation}}

\newcommand{\res}{\ensuremath{\,\textsf{\small \upshape L}\,}}

\newcommand{\spt}{{\rm spt}\,}

\newtheorem{theorem}{Theorem}[section]
\newtheorem{lemma}[theorem]{Lemma}
\newtheorem{corollary}[theorem]{Corollary}

\newtheorem{proposition}[theorem]{Proposition}
\theoremstyle{definition}

\theoremstyle{remark}
\newtheorem{remark}[theorem]{Remark}
\numberwithin{equation}{section}
\def\uclhome{@ucl.ac.uk}

\begin{document}

\title[Optimal isoperimetric inequalities in Cartan-Hadamard manifolds]{Optimal isoperimetric inequalities for surfaces in any codimension in Cartan-Hadamard manifolds}
% Information for first author
\author{Felix Schulze}
%\thanks{}
\address{Felix Schulze: 
  Department of Mathematics, University College London, 25 Gordon St,
  London WC1E 6BT, UK}
% Current address
\curraddr{}
\email{f.schulze\uclhome}

% General info
\subjclass[2000]{}

\dedicatory{}


\begin{abstract}  Let $(M^n,g)$ be simply connected, complete, with non-positive sectional curvatures, and $\Sigma$ a 2-dimensional closed integral current (or flat chain mod 2) with compact support in $M$. Let $S$ be an area minimising integral 3-current (resp.~flat chain mod 2) such that $\partial S = \Sigma$. We use a weak mean curvature flow, obtained via elliptic regularisation, starting from $\Sigma$, to show that S satisfies the optimal Euclidean isoperimetric inequality: $ 6 \sqrt{\pi}\, \mathbf{M}[S] \leq (\mathbf{M}[\Sigma])^{3/2} $. We also obtain an optimal estimate in case the sectional curvatures of M are bounded from above by $-\kappa < 0$ and characterise the case of equality. The proof follows from an almost monotonicity of a suitable isoperimetric difference along the approximating flows in one dimension higher and an optimal estimate for the Willmore energy of a 2-dimensional integral varifold with first variation summable in $L^2$.

\end{abstract}

\maketitle

\section{Introduction}
The classic Euclidean isoperimetric inequality states that for any bounded open set $\Omega \subset \R^{n+1}$ with sufficiently regular boundary it holds that
\begin{equation} \label{eq:iso.eucl}
|\Omega| \leq \gamma(n+1) |\partial \Omega|^\frac{n+1}{n}
 \end{equation}
where $|\Omega|$ is the Lebesgue measure of $\Omega$, $|\partial \Omega|$ is a suitable notion of measure of the boundary and $\gamma(n+1) = (n+1)^{-(n+1)/n} \omega_{n+1}^{-1/n}$, where $\omega_{n+1}$ is the measure of the Euclidean unit $(n+1)$-ball. Furthermore, equality is attained if and only if $\Omega$ is a ball.

It is a natural question if there is a corresponding statement in higher codimension. This  was answered by Almgren in \cite{Almgren86}, which can be losely stated as follows.

\begin{theorem}[Almgren] \label{thm:almgren} Corresponding to each $m$-dimensional closed surface $T$ in $\R^{n+1}$ there is an $(m+1)$-dimensional surface $Q$ having $T$ as boundary such that 
\begin{equation}\label{eq:isoopt}
|Q| \leq \gamma(m+1) |T|^\frac{m+1}{m}
\end{equation}
with equality if and only if $T$ is a standard round $m$ sphere (of some radius) and $Q$ is the corresponding flat $m+1$ disk.
 \end{theorem}
Here $|Q|$ and $|T|$ denote the areas in dimension $m+1$ and $m$ respectively. To be more precise, the notion 'surface' can be taken to be real rectifiable currents, real polyhedral chains, integral currents or flat chains mod $\nu$.

A further natural question to extend \eqref{eq:iso.eucl} is to ask on which Riemannian manifolds $(M^n,g)$ does the Euclidean isoperimetric inequality hold. A natural candidate are Cartan-Hadamard manifolds $(M^n,g)$; that is complete, simply connected Riemannian manifolds with non-positive sectional curvature. We will denote the space of such manifolds with sectional curvatures bounded from above by $-\kappa$ for $-\kappa \leq 0$ by $\mathcal{CH}(n,-\kappa)$. The following conjecture appeared in Aubin \cite{Aubin76}, Burago-Zalgaller \cite{BuragoZalgaller88} and Gromov \cite{Gromov99}.

{\bf Conjecture:} Let $(M^n,g)\in \mathcal{CH}(n,0)$.  Then the Euclidean isoperimetric inequality \eqref{eq:iso.eucl} holds on $(M,g)$.

The conjecture can be strengthened by asking that if $(M,g) \in \mathcal{CH}(n,-\kappa)$ for $-\kappa <0$, then the isoperimetric inequality of the model space with sectional curvatures equal to $-\kappa$ holds. This has been proven for $n=2$ and $\kappa =0 $ by Weil \cite{Weil26}, for $n=2$ and $\kappa \neq 0 $ by Bol \cite{Bol41} , for $n=3$ and $-\kappa \leq 0$ by Kleiner \cite{Kleiner92} and for $n=4$ and $\kappa = 0$ by Croke \cite{Croke84}.  For a more detailed overview of the history of the conjecture and a partial extension for $n=4$ and $\kappa \neq 0$ see Kloeckner-Kuperberg \cite{KloecknerKuperberg13}. The remaining cases are open. Using a variant of mean curvature flow we gave an alternative proof of Kleiner's result in \cite{HkIso}.

The question if an isoperimetric estimate as in \eqref{eq:isoopt} holds for any dimension $m \leq n-1$ with a  non-optimal constant, depending only on $m$, was first resolved for Euclidean space by Federer-Fleming \cite{FedererFleming60} and by Gromov \cite{Gromov83} for a certain class of complete Riemannian manifolds, including Cartan-Hadamard manifolds. For further extensions to metric spaces see also \cite{AmbrosioKirchheim00, Gromov83, Wenger05}.

The main result in this article is an extension of Almgren's result to $2$-dimensional surfaces in Cartan-Hadamard manifolds with arbitrary codimension.

\begin{theorem}\label{thm:mainthm}
Let $(M^{n},g) \in \mathcal{CH}(n,-\kappa), -\kappa\leq 0, n\geq 3$ and $\Sigma\subset M$ an integral 2-current (flat chain mod 2) with compact support such that $\partial \Sigma =0$. Let $S$ be an area minimising integral 3-current (flat chain mod 2) such that $\partial S = \Sigma$. Then
\begin{equation}\label{eq:main.1}
\mathbf{M}[\Sigma] \geq \ch^2(\partial B_r) \ ,
\end{equation}
where $B_r$ is a geodesic ball in the $3$-dimensional model space with sectional curvatures equal to $-\kappa$ and radius $r$ such that $\ch^3(B_r) = \mathbf{M}[S]$.
\end{theorem}

Here $\bfM [\,\cdot\,]$ denotes the mass of a current and $\mathcal{H}^k$ the $k$-dimensional Hausdorff measure in the model space. The corresponding isoperimetric inequality in a Cartan-Hadamard manifold for closed (smooth) curves bounding a smooth minimal surface in any codimension follows from general results of Reshetnyak \cite{Resh-majorization, Resh-isop}.

We can also characterise the equality case.

\begin{theorem}\label{thm:mainthm_2}
If equality in \eqref{eq:main.1} is attained, then $\Sigma$ is a smooth embedded $2$-sphere, has unit density, its mean curvature vector has constant length and $\Sigma$ is totally umbilic. Furthermore, it bounds a totally geodesic embedded $3$-ball $S$, with the mean curvature vector of $\Sigma$ proportional to the unit conormal of $S$ at every point in $\Sigma$.  $S$ is isometric to a geodesic ball in the $3$-dimensional model space such that the mean curvature of the boundary coincides with the one of $\Sigma$.
\end{theorem}

A central ingredient in the proof is an optimal lower bound for the Willmore energy of an integral $2$-varifold with compact support. 

\begin{theorem} \label{lem:lowerbound} Let $(M^{n},g) \in \mathcal{CH}(n,-\kappa), -\kappa \leq 0, n\geq 3$ and $\Sigma^2$
 be an integral 2-varifold in $M$ with compact support and that the first variation of $\Sigma$ is summable in $L^2(\mu)$. Then
\begin{equation}\label{eq:main.3}
 \int_\Sigma |\mathbf{H}|^2 \, d\mu \geq 16\pi + 4 \kappa |\Sigma| \ , 
 \end{equation}
where $\mathbf{H}$ is the weak mean curvature of $\Sigma$. If equality is attained, then $\Sigma$ is a smooth embedded $2$-sphere, has density one, the mean curvature vector has constant length and $\Sigma$ is totally umbilic. Furthermore, it bounds a totally geodesic embedded $3$-ball $S$, with the mean curvature vector of $\Sigma$ proportional to the unit conormal of $S$ at every point in $\Sigma$.  $S$ is isometric to a geodesic ball in the 3-dimensional model space such that the mean curvature of the boundary coincides with the one of $\Sigma$.
\end{theorem}

This estimate for $n=3$ and $\kappa =0$ appeared already in \cite[Lemma 6.7]{HkIso}. In Euclidean space the estimate can be found in work of Simon \cite{Simon93}, and follows rather directly from the usual calculations leading to the monotonicity formula. The characterisation of the equality case in an Euclidean ambient is given by Lamm-Sch\"atzle \cite{LammSchaetzle14} together with a stability result. For smooth surfaces in an Euclidean ambient space this is the well known Li-Yau estimate, \cite{LiYau82}. For sufficiently regular surfaces in codimension one which are outward minimising, an analogous estimate following from the Gauss-Bonnet formula is central in the argument of Kleiner \cite{Kleiner92} (see also the alternative proof of Ritor\'e \cite{Ritore05}, which does not require the condition of outward minimising).

\begin{remark}
For $(M,g) \in  \mathcal{CH}(n,0)$ and $\Sigma$ an integral $m$-varifold in $M$, one can use the variant of the Michael-Simon Sobolev inequality \cite{MichaelSimon73} for Riemannian manifolds by Hoffman-Spruck \cite{HoffmanSpruck74, HoffmannSpruck75} (which carries over to the setting of integral varifolds) to get an estimate
$$|\Sigma|^\frac{m-2}{m} \leq C(m) \int_\Sigma |\mathbf{H}|^2 \, d\mu $$
where $C(m)$ depends only on $m$. This constant is not optimal, but the proof of Theorem \ref{thm:mainthm} carries over to any dimension and codimension, yielding a non-optimal inequality for integral currents or flat chains mod 2 as in \eqref{eq:isoopt} with a constant only depending on $m$. Alternatively, restricting to an open, precompact set $K \subset M$,  a direct comparison with Euclidean space gives a non-optimal isoperimetric inequality, where the constant depends on $(M,g)$ and $K$, see Lemma \ref{lem:iso}.
\end{remark}

We give a first outline of the idea of the proof of Theorem \ref{thm:mainthm} for $\kappa =0$. Assume $ n\geq 3, (M^n,g) \in \mathcal{CH}(n,0)$ and that $\Sigma$ is an orientable, closed, smooth, 2-dimensional submanifold of $M$. Let $(\Sigma_t)_{0\leq t<T}$ be its smooth evolution by mean curvature flow with $\Sigma_0 = \Sigma$. Assume further that there exists a smooth family 
\begin{equation}\label{eq:intro.0}
(S_t)_{0\leq t<T}
\end{equation}
of minimal 3-dimensional (immersed) submanifolds in $M$ such that $ \partial S_t = \Sigma_t$,
and let $X_t$ be the variation vectorfield along $(S_t)_{0\leq t<T}$. The first variation formula then implies that
$$\frac{d}{dt} |S_t| = \int_{S_t} \text{div}_{S_t}(X)\, d\ch^3 = \int_{\Sigma_t} \langle X, \mathbf{n}\rangle\, d\ch^2 =\int_{\Sigma_t} \langle \bfH, \mathbf{n}\rangle\, d\ch^2\ ,$$
where $\mathbf{n}$ is the unit conormal of $S_t$ along $\partial S_t = \Sigma_t$ and $\bfH$ is the mean curvature vector of $\Sigma_t$. Similar to \cite{HkIso}, we consider the isoperimetric difference
$$I_t=|\Sigma_t|^{3/2} - 6 \sqrt{\pi} |S_t|$$
and compute, using \eqref{eq:main.3} in the second line,
\begin{equation*}\begin{split}
 - \frac{d}{dt} |S_t| &= -   \int_{\Sigma_t} \langle \bfH, \mathbf{n}\rangle\, d\ch^2 \leq  \int_{\Sigma_t} |\bfH| \, d\ch^2 \\
 &\leq \left(\int_{\Sigma_t} |\bfH|^2 \, d\ch^2\right)^{1/2} |\Sigma_t|^{1/2} \cdot \frac{1}{4 \sqrt{\pi}} \left(\int_{\Sigma_t} |\bfH|^2 \, d\ch^2\right)^{1/2} \\
 &\leq \frac{1}{4 \sqrt{\pi}} |\Sigma_t|^{1/2}\int_{\Sigma_t} |\bfH|^2 \, d\ch^2 = - \frac{1}{6 \sqrt{\pi}} \frac{d}{dt} |\Sigma_t|^{3/2}
 \end{split}
 \end{equation*}
 and thus $\tfrac{d}{dt} I_t \leq 0$. If the flow  $(\Sigma_t)_{0\leq t<T}$ and the family $(S_t)_{0\leq t<T}$ exists long enough such that $\lim_{t\ra T} |S_t| =0$, this shows that
$$ |S_0| \leq \frac{1}{6\sqrt{\pi}} |\Sigma_0|^{3/2}\ .$$
But in general it can't be expected that the flow does not develop singularities before the spanning volume goes zero. It is also not clear why a sufficiently regular family $(\Sigma_t)_{0\leq t<T}$ should exist. To be able to evolve through singularities we would like to work with a weak solution of mean curvature flow, in our case the most suitable one seems to be a Brakke flow. But there are only very little regularity results for higher codimension, even sudden vanishing is possible. Furthermore, it is not clear to us how to construct a sufficiently regular family of spanning minimal surfaces such that the above monotonicity calculation can be performed. 

To circumvent this problem we work with Ilmanen's elliptic regularisation scheme \cite{Ilmanen}. In this work Ilmanen combines the elliptic regularisation approach of Evans-Spruck \cite{EvansSpruckI}  in codimension one with the moving varifold solutions of Brakke \cite{Brakke} to construct Brakke flow solutions with special properties. Treating all surfaces as if they were smooth and avoiding some of the technical details, we give an overview of the argument to prove Theorem \ref{thm:mainthm} for $\kappa = 0$.

Let $\Sigma_0 \subset M$ be an integral $2$-current with compact support such that $\partial \Sigma_0 = 0$.  We consider $\Sigma_0 \subset M\times \{0\} \subset M \times \R$ and denote by $z$ the coordinate in the additional $\R$-direction and $\tau$ the corresponding unit vector. Ilmanen's elliptic regularisation scheme yields a sequence $\eps_i>0,\, \eps_i \ra 0$, a sequence of integral $3$-currents $P^i$ such that $\partial P^i = \Sigma_0$, which yield translating solutions to mean curvature flow in $M \times \R$ via
$$P^i(t) = P^i - \frac{t}{\eps_i}\tau\ .$$ 
Let $\{\mu^i_t\}_{t \in \R}$ be the corresponding family of Radon measures. This sequence of flows converges as $i \ra \infty$ to a limiting Brakke flow $\{\bar{\mu}_t\}_{t \geq 0}$ which is invariant in $z$-direction, starting at $\Sigma_0\times \R$. The Brakke flow $\{\mu_t\}_{t \geq 0}$ starting at $\Sigma_0$ is then obtained via slicing $\{\bar{\mu}^i_t\}_{t \geq 0}$ at height $z=\text{const}$. Additionally, the sequence 
$$T^i = \kappa_{\eps_i}(P^i)\, ,$$
where $\kappa_{\eps_i}(x,z) = (x,\eps_i z)$, converges to a current $T \subset M\times \R^+$ such that $\partial T = \Sigma_0$. Furthermore, 
$$ \mu_t \geq \mu_{T_t}$$
where $\mu_{T_t}$ is the mass measure associated to the slice $T_t$ of $T$ at height $z=t$. The current $T$ is called the \emph{undercurrent} of the flow $\{\mu^i_t\}_{t \geq 0}$. Treating the $z$-direction as time, it can be helpful to think of $T$ as the space-time track of the flow $\{\mu_t\}_{t \geq 0}$, after taking into account possible cancellations. Furthermore for all $t>0$
\begin{equation}\label{eq:intro.1}
 P^i(t) \ra \pi(T_t) \times \R
\end{equation}
as $i \ra \infty$, where $\pi: M \times \R \ra M$ is the projection on the first factor. 

We choose $S_0$, a mass-minimising integral $3$-current with $\partial S_0 = \Sigma$ and $S^i$ mass-minimising integral $4$-currents in $M\times \R$ such that 
$$\partial S^i = P^i - S_0$$
and denote 
$$ S^i(t) = S^i - \frac{t}{\eps_i}\tau\, .$$
This family will serve as a family of minimal surfaces approximating the family \eqref{eq:intro.0} considered in the smooth monotonicity calculation. Note that the variation vectorfield of this family is just given by $X = - \eps_i^{-1}\tau$, which makes the monotonicity calculation for \eqref{eq:intro.3} feasible. 

Let $l >1$. We choose $\varphi_l \in C^2_c(\R)$ such that $0 \leq \varphi_l \leq 1/l$ with $\varphi_l = 1/l$ on $[2,l+2]$, $\varphi_l = 0$ on $(0,\infty) \setminus [1,l+3]$. We define the approximate area and volume by 
 $$ A^i_t = \int \varphi_l \, d\mu^i_t \qquad \text{and} \qquad V^i_t:= \int \varphi_l\, d\mu^{S,i}_t\, .$$
The averaging function $\varphi_l$ takes into account that in the limit $i \ra \infty,\, P^i(t)$ becomes vertical, and thus $A^i_t$ approximates $\mu_t(M)$.  For $t$ fixed and $i \ra \infty$ we expect that $S^i(t)$ has a similar behaviour  and thus $V^i_t$ approximates the measure of a family as in \eqref{eq:intro.0}.

In Euclidean space shrinking spheres with radius $R(t) = \sqrt{R^2-2mt}$ act as barriers for integral $m$-Brakke flows from the inside and from the outside. Using the properties of the Hessian of the distance function to a point $p$ in a Cartan-Hadamard manifold $M^n$ one can show that this remains true as barriers from the outside, and thus the flow $\{\mu^i_t\}_{t \geq 0}$ has a finite maximal existence time $T_\text{max} \leq R^2/4$, provided $\Sigma_0 \subset B_R(p)$. 

To see that 
\begin{equation}\label{eq:intro.2}
V^i_t < \eps \text{ for } t \text{ close to }T_\text{max}, \, l \geq l_0\text{ and }i \text{ sufficiently large,}
\end{equation}
one can use the future space-time track of the flow as a competitor: motivated by the fact that an estimate for the volume traced out by a mean curvature flow is given by the $L^1$-norm in time of the mean curvature vector, and the natural estimate
$$ \int_0^{T_\text{max}} |\mathbf{H}|^2 \, d\mu_t\, dt \leq  \bfM [\Sigma_0] \, ,$$
where $\bfM [\Sigma_0]$ is the measure of $\Sigma_0$, Ilmanen shows that
$$ \bfM [ \pi(T\cap \{z\geq t\})] \leq (T_\text{max}-t)^{1/2} \bfM [\Sigma_0]\ .$$
Noting that $\partial (\pi(T\cap \{z\geq t\})) = \pi(T_t)$ and recalling \eqref{eq:intro.1} we can use $\pi(T\cap \{z\geq t\}) \times \R$, up to a small error, as a competitor to $S^i(t)$ to achieve \eqref{eq:intro.2}.

For the monotonicity calculation we consider the approximate isoperimetric difference
\begin{equation}\label{eq:intro.3}
I^i_t=|A^i_t|^{3/2} - 6 \sqrt{\pi} |V^i_t|\, ,
\end{equation}
and show that this quantity is monotone in the limit as $l\ra \infty$ and $i \ra \infty$ between $t_0=0$ and $0<t_1<T_\text{max}$. To see this we show that the error terms in the time derivative of \eqref{eq:intro.3} are controllable and combine the property that $P^i(t)$ becomes vertical with the estimate \eqref{eq:main.3} and the lower semicontinuity of the $L^2$-norm of the mean curvature. Together with \eqref{eq:intro.2} this yield that
$$ (\bfM [\Sigma_0])^{3/2} \geq 6 \sqrt{\pi}\, \bfM [S_0]\, .$$

{\bf Structure of the paper.} In \S \ref{sec:elliptic} we recall Ilmanen's ellitpic regularisation scheme \cite{Ilmanen} and show the improved approximation \eqref{eq:intro.1}.  The barrier argument  and a comparison principle due to B.~White yield the estimate on the maximal existence time. We also prove a positive lower estimate on the maximal existence time for the limiting Brakke flow. 

An essential ingredient in controlling the error terms when showing the almost monotonicity of the approximate isoperimetric difference is to know that
\begin{equation}\label{eq:intro.4}
 S^i \ra S_0 \times [0,\infty)\, 
 \end{equation}
as $i \ra \infty$. To achieve this we first assume that $S_0$ is the \emph{unique} mass-minimising current spanning $\Sigma_0$. Using this assumption, we show in \S \ref{sec:initial attainment} that \eqref{eq:intro.4} holds. We later show that by perturbing $\Sigma_0$ slightly we can assume that $\Sigma_0$ bounds only one mass-minimising current. We also give uniform local area bounds for $S^i$. 

In \S \ref{sec: area vanishing final} we prove \eqref{eq:intro.2}. 

In \S \ref{sec: mon calc} we compute the time derivative of the approximate isoperimetric difference and show that the error terms are controllable in the limit $i \ra \infty$. We use a lower semi-continuity argument together with \eqref{eq:main.3} to prove Theorem \ref{thm:mainthm}. We also show that we can treat the case of equality, Theorem \ref{thm:mainthm_2}, using the characterisation of equality in Theorem \ref{lem:lowerbound}.

In \S \ref{sec: lower bound Willmore} we prove Theorem \ref{lem:lowerbound}.

In the appendix we collect several results needed in the prequel. We show that the mass minimising currents $S^i$ are strongly stationary and that there is a non-optimal isoperimetric inequality in any dimension and codimension in a Cartan-Hadamard manifold. Furthermore, we recall White's avoidance principle for Brakke flows and show how unique continuation for minimal surfaces in any codimension follows from work of Kazdan.

{\bf Acknowledgements.} We are grateful to C.~Bellettini and B.~White for several inspiring and  helpful discussions.

\section{Elliptic regularisation}\label{sec:elliptic}

We employ Ilmanen's elliptic regularisation
scheme \cite{Ilmanen} to construct a Brakke flow starting at $\Sigma$. We
recall the construction of Ilmanen, adapted to our setting, and its properties
needed in the sequel.

\begin{theorem}[\cite{Ilmanen}, \S 8.1]\label{matching-motion}
  Let $T_0$ be local
integral $m$-current in $(M^{m+k},g)$ with $\partial T_0 = 0$ and
finite mass ${\bf M}[T_0]<\infty$. Then there exists a local integral
$(m+1)$-current $T$ in $M\times [0,\infty)$ and a family
$\{\mu_t\}_{t\geq 0}$ of Radon measures on $M$ such that
\begin{itemize}
\item[$(i)$] (a) $\partial T = T_0$\\[1ex]
  (b) ${\bf M}[T_B]$, where $T_B = T\res (M\times B),\ B\subset
  \R$, is absolutely continuous with respect to $\mathcal{L}^1(B)$.\\[-2ex]
\item[$(ii)$] (a) $\mu_0=\mu_{T_0}, {\bf M}[\mu_t]\leq {\bf M}[\mu_0]$
  for $t>0$.\\[1ex]
(b) $\{\mu_t\}_{t\geq 0}$ is an integral $n$-Brakke flow.\\[-2ex]
\item[$(iii)$] $\mu_t\geq \mu_{T_t}$ for each $t\geq 0$, where $T_t$ is
  the slice $\partial(T\res (M^{m+k}\times [t,\infty))$.  
\end{itemize}
\end{theorem}

We outline the main steps of the proof. Ilmanen constructs local
integral $(m+1)$-currents $P_\eps$ in $M^{m+k}\times \R$ that
minimize the
elliptic translator functional
$$ I^\eps[Q] = \frac{1}{\eps}\int e^{-z/\eps}\, d\mu_Q(x,z)\, ,$$ 
where $z$ is the coordinate in the additional $\R$-direction, subject
to the boundary condition
$$\partial Q = T_0\, ,$$
and $M^{m+k}$ is identified with the height zero slice in
$M^{m+k}\times \R$. Note that $I^\eps$ is the area functional for the
metric $\bar{g} = e^{-2z/((m+1)\eps)}(g \oplus dz^2)$, where $g \oplus
dz^2$ is the product metric on $M^{m+k}\times \R$. 

The associated Euler-Lagrange equation implies that the family of
Radon measures $\mu^\eps_t=\mu_{P^\eps_t}$ corresponding to
$$P^\eps (t) = (\sigma_{-t/\eps})_\# (P^\eps)$$ 
for $0\leq t<\infty$, where
$\sigma_{-t/\eps}(x,z)=(x,z-t/\eps)$, is a downward translating 
integral $(m+1)$-Brakke flow on the relatively open subset
$W^\eps := \{(x,z,t)\, :\, z>-t/\eps,\ t\geq 0\}$ of space-time
$(M^{m+k}\times \R)\times [0,\infty)$. 

Ilmanen's compactness theorem for Brakke flows implies that there is a
sequence $\eps_i\ra 0$ such that $\{\mu^{\eps_i}_t\}_{t\geq 0}$ converges
to a Brakke flow $\{\bar{\mu}_t\}_{t\geq 0}$ on
space-time. Furthermore, Ilmanen shows that $\bar{\mu}_0 = \mu_{T_0\times \R}$ and
$\bar{\mu}_t$ is invariant in the $z$-direction, which yields the
desired solution $\{\mu_t\}_{t\geq 0}$ via slicing. 

The integral current $T$ is constructed via considering a
subsequential limit of $T^{\eps_i}:=(\kappa_{\eps_i})_\#(P^{\eps_i})$ where
$\kappa_{\eps_i}(x,z)=(x,\eps_i z)$, which can be seen as an approximation
to the space-time track of $\{\mu_t\}_{t\geq 0}$ where now the
$z$-direction is considered as the time direction. Point $(iii)$
above verifies this interpretation.

Recall that for $s\geq 0$ we define the following slices by the height function $z$:
$$ P^{\eps_i}_s = \partial(P^{\eps_i} \res (M \times [s,\infty))$$
and similarly
$$ T_s = \partial(T \res (M \times [s,\infty))\, .$$

We note the following estimates from \cite{Ilmanen}. 
\begin{proposition}[Ilmanen] \label{prop:estimatesIlmanen}
The following estimates hold:  For any measurable subset $A \subset \R$
\begin{equation}\label{eq:massbound1}
\mathbf{M}[P^\eps_A] \leq (|A| + \eps) \mathbf{M}[T_0]
\end{equation}
where $P_A = P\res (M\times A)$ and $|A|$ is the measure of $A$. \\[1ex]
Let $\pi: M \times \R \ra M$ be the projection onto $M$. Then for any measurable subset $B \subset \R$
\begin{equation}\label{eq:massbound2} \mathbf{M}[\pi_\#(T^\eps_B)]  \leq (|B| + \eps^2)^{1/2} \mathbf{M}[T_0]\ .
\end{equation}
In particular in the flat metric distance
\begin{equation}\label{eq:massbound3}
\text{{\rm dist}}(\pi_\#(T^\eps_t), \pi_\#(T^\eps_{t+\delta})) \leq (\delta + \eps^2)^{1/2} \mathbf{M}[T_0]\, .
\end{equation}
Furthermore,
\begin{equation}\label{eq:massbound4}
\mathbf{M}[(T^\eps_B)] \leq \big(|B| + \eps^2+ (|B| + \eps^2)^{1/2}\big) \mathbf{M}[T_0]\ .
\end{equation}
In particular in the flat metric distance
\begin{equation}\label{eq:massbound5}
\text{{\rm dist}}(T^\eps_t, T^\eps_{t+\delta}) \leq \big(\delta + \eps^2+ (\delta + \eps^2)^{1/2}\big) \mathbf{M}[T_0]\, .
\end{equation}

\end{proposition}
For details see \S 5.1 -- \S 5.3 in  \cite{Ilmanen}.

One can use the $C^{1/2}$-continuity of  $(T^\eps_t)$ to show the following improved approximation property.
 
\begin{lemma} \label{lem:weakconv} 
We have
$$ P^{\eps_i} = P^{\eps_i}(0) \ra  T_0\times [0,+\infty)\ , $$
and for $t>0$
$$  P^{\eps_i}(t) \ra \pi_\#(T_t)\times \R $$
in the sense of currents.  
\end{lemma}
\begin{proof}
Fix $t\geq 0$. By \eqref{eq:massbound1} we can assume, up to a subsequence, that $P^{\eps_i}(t) \ra P'$. 
Recall that $T^{\eps_i} \ra T$ and thus $T^{\eps_i} \res (M \times [t, \infty)) \ra T\res (M \times [t, \infty))$ for any $ t\geq 0$. Taking boundaries this yields
$$ T^{\eps_i}_t \ra T_t\, .$$
Note that $T^{\eps_i}_t = (\kappa_{\eps_i})_\# P^{\eps_i}_{t/\eps_i}$. This implies that
$$  (\kappa_{\eps_i})_\# (P^{\eps_i}_{t/\eps_i + s}) = T^{\eps_i}_{t+ \eps_i s}\ .$$
Using \eqref{eq:massbound3} this yields that for $t=0$ and $s\geq 0$ or $t>0$ and any $s\in \R$
$$ \pi_\#(P^{\eps_i}_{t/\eps_i + s}) \ra \pi_\#(T_t) \, .$$
This yields that for any any $s\in [0,\infty)$
$$ \pi_\#(P^{\eps_i}_s) \ra \pi_\#(T_0) $$
and thus
$$ \pi_\#(P'_s) = \pi_\#(T_0)\ .$$
Furthermore, by \eqref{eq:massbound2} we have for any $0\leq s_1 < s_2$ that
$$ \mathbf{M}[\pi_\#(P^{\eps_i}_{[s_1,s_2]})] = \mathbf{M}[\pi_\#(T^{\eps_i}_{[\eps_i s_1, \eps_i s_2]})] \leq (\eps_i(s_2-s_1) + \eps_i^2)^{1/2} \mathbf{M}[T_0] $$
and thus 
$$\mathbf{M}[\pi_\#(P'_{[s_1,s_2]})] = 0\ .$$
This yields that $\tfrac{\partial}{\partial z}$ is $\mathcal{H}^{m+1}-a.e.$ tangential to $P'$. By the coarea-formula this
implies that 
$$ P'= \pi_\#(T_0)\times [0,\infty)\ .$$
For $t>0$, we obtain that for any any $s\in \R$
$$ \pi_\#((P^{\eps_i}(t))_s) = \pi_\#(((\sigma_{-t/\eps_i})_\#(P^{\eps_i}))_s) = \pi_\#(P^{\eps_i}_{t/\eps_i + s}) \ra \pi_\#(T_t) $$
and by the same argument as earlier that
$$ P^{\eps_i}(t) \ra P'=\pi_\#(T_t)\times \R\ .$$
\end{proof}

We will in the following always assume that $(M, g) \in \mathcal{CH}(n,0)$. We consider the local integral 3-currents $P^\eps \subset M\times [0,\infty)$ constructed in the previous section, such that $\partial P^\eps = \Sigma_0$. We choose a sequence $\eps_i \ra 0$ such that as in the proof of Theorem \ref{matching-motion}, we have $\{\mu^{\eps_i}_t\}_{t\geq 0}$ converging
to a Brakke flow $\{\bar{\mu}_t\}_{t\geq 0}$ which is invariant in the $z$-direction (which we can w.l.o.g. ~assume is true for all $t$) and $T^{\eps_i} \ra T$. Let $\{\mu_t\}_{t\geq 0}$ be the Brakke flow starting at $\Sigma_0$ obtained from $\{\bar{\mu}_t\}_{t\geq 0}$ via slicing in $z$-direction. We denote the maximal existence time of the constructed Brakke flow $\{\mu_t\}_{t\geq 0}$, by
 $$ T_\text{max} = \inf_{t>0}\{t\, |\, \mu_t = 0\}\, . $$
Note that by the monotonicity of the total measure we have $\mu_t(M)>0$ for all $t< T_\text{max}$ and $\mu_t(M)=0$ for all $t>T_\text{max}$. Under the present restrictions on the geometry of $M$ we obtain an upper bound for the maximal existence time.

\begin{lemma}\label{lem:maxexist}
 Assume $(M, g) \in \mathcal{CH}(n,0)$. Let $p_0 \in M$ and $\spt \Sigma_0 \subset B_R(p_0)$. Then $\spt \mu_t \subset B_{r(t)}(p_0)$ where $r(t) = \sqrt{R^2 - 4t}$. The maximal existence time $ T_\text{max}$ of the constructed brakke flow $\{\mu_t\}_{t\geq 0}$
 is bounded from above by $R^2/4$. Furthermore,
 $$ \spt P^\eps \subset \big\{(p,z)\, |\, 0 \leq z \leq \eps^{-1} \big(R^2 + o(1) - d(p,p_0)^2\big)/4\big\}\, .$$
\end{lemma}

\begin{proof}
Let $r(p):= d(p,p_0)$. Since $M$ is complete and has non-positive sectional curvature we have
$$ \nabla^2 r \geq r^{-1}(\text{id} - \nabla r \otimes \nabla r )\, .$$
Consider $0<\alpha<n$ and the function
$$u(p,t) = r^2 +2\alpha t\ .$$
Then with the notation as in Theorem \ref{thm:barrier} we see that
$$\frac{\partial u}{\partial t} - \text{tr}_2 \nabla^2 u = 2\alpha -  2 \text{tr}_2( r \nabla^2r + \nabla r \otimes \nabla r) \leq 
2\alpha -  2 \text{tr}_2\text{id} < 0\, ,$$
and thus by Theorem \ref{thm:barrier}
$$ u(x,t) \leq R^2 $$
on $\spt \mu_t$. Letting $\alpha \ra n$ this implies the first two statements. To obtain the height bound observe that by Huisken's monotonicity formula (with a suitable local modification due to the non-flat background)
the support of the Brakke flow $(\mu^\eps_t)_{t\geq 0}$ converges in Hausdorff distance to the support of $(\bar{\mu}_t)_{t\geq 0}$. 
\end{proof}

Let $S_0$ be an area-minimising 3-current in $M$ such that
$$ \partial S_0 = T_0 .$$ 
Note that geodesic spheres in $M$ are convex, thus by the convex hull property we have that the support of $S_0$ is compact. We then also obtain a lower bound on the maximal existence time.

\begin{lemma}\label{lem:minexist}
  Assume that $\bfM [S_0] >0$. Then
there exists $\delta = \delta(\bfM [\Sigma_0], \bfM [S_0]) >0$ and $\eta= \eta(\bfM [\Sigma_0], \bfM [S_0])>0$ such that
$$ \mu_{t}(M) \geq \eta $$
for all $0\leq t < \delta$.
\end{lemma}
\begin{proof}
 Let $T$ be the undercurrent of the flow $\{\mu_t\}_{t\geq 0}$. Note that by \eqref{eq:massbound4} we have the estimate 
$$\bfM[\pi_\#(T_{[t,s]})] \leq (|s-t| + |s-t|^{1/2}) \mathbf{M}[\Sigma_0] $$
and thus for any mass-minimising integral $3$-current $S_t$ spanning $T_t$ we can estimate
$$ \mathbf{M}[S_t] \geq  \mathbf{M}[S_0] - \bfM[\pi_\#(T_{[0,s]})] \geq  \mathbf{M}[S_0] -  2\, t^{1/2} \mathbf{M}[\Sigma_0]
\geq \frac{\mathbf{M}[S_0]}{2}\, ,$$
for $t\leq \delta$. By Lemma \ref{lem:iso} have
$$ \mathbf{M}[T_t] \geq \eta >0 $$
for all $0\leq t < \delta$ and all $k$ sufficiently large. The claim then follows from Theorem \ref{matching-motion} $(iii)$.
\end{proof}

\section{Attainment of initial spanning surface}\label{sec:initial attainment}

We can w.l.o.g. assume that $\bfM [S_0] >0$. We will for the moment work with the following\\[2ex]
{\bf Assumption:} We assume $S_0\subset M$ is the unique area-minimising 3-current spanning $\Sigma_0$.\\[2ex]
We will later verify that in general one can perturb $\Sigma_0$ slightly such that the uniqueness assumption is satisfied.

Let $S^\eps$ be area-minimising 4-currents in $M\times [0,\infty)$ such that
$$\partial S^{\eps_i} = P^{\eps_i} - S_0\ .$$
In the remaining part of this section we aim to show that
\begin{equation}\label{eq:initialattainment}
S^{\eps_i} \ra S_0\times [0,\infty)
\end{equation}
as $\eps_i \ra 0$.

\begin{lemma} Let $(M,g) \in \mathcal{CH}(n,0)$. For all $p\in M\times [0,\infty)$ and $r \geq 1$ it holds
  \begin{equation}\label{eq:massbound2b}
  \mathbf{M}[S^\eps\res B_r(p)] \leq \omega_{4}\Big(r^2+ \frac{\eps}{3}r\Big) \mathbf{M}[\Sigma] + \frac{\omega_4}{4} \mathbf{M}[S_0]\ .
  \end{equation}
\end{lemma}

\begin{proof}  The proof of the classical monotonicity formula in $\R^{n}$ relies on the fact that the position vectorfield $X(x,x_0) = x-x_0$ satisfies $\text{div}_{T}(X) = k$, where $T$ is an $k$-dimensional subspace of $T_x\R^{n}$.  As in the proof of Theorem \ref{lem:lowerbound} we replace the position vectorfied $X(x,x_0)$ by
$$X_p(x) := r \bar{\nabla}r $$
where $r = d(x,p)$. As in the proof there, one can check that this vectorfield satisfies
\begin{equation}\label{eq:mon.1}
 \text{div}_{T}(X_p) \geq k\, ,
\end{equation}
 where $T$ is an $k$-dimensional subspace of $T_x ( M\times \R)$. Using that $S^\eps$ is strongly stationary, see Lemma \ref{lem:div thm}, one obtains as in the proof of the monotonicity formula, compare \cite{Simon}, that
\begin{equation*} 
 	\begin{split}
	\frac{d}{d\rho} \bigg(\frac{\mathbf{M}[S^\eps\res B_\rho(p)]}{\omega_4 \rho^4} \bigg) &\geq \frac{d}{d\rho} \int_{ B_\rho(p)} \frac{|\bar{\nabla}
^\perp r|^2}{r^n} \, d\mu_{S^\eps} + \rho^{-5} \int_{B_\rho(p)} \langle X_p(x), \mathbf{H} \rangle\, d\mu_{S^\eps}\\
&\ \ \  -  \rho^{-5} \int_{B_\rho(p)} \langle X_p(x), \mathbf{n}_S \rangle\, d\mu_{\partial S^\eps}\, .
 	\end{split}
\end{equation*}
Together with estimate \eqref{eq:massbound1} this yields for $\rho \geq 1$ and $0<\eps<1$
\begin{equation*} 
\begin{split}
\frac{d}{d\rho} \bigg(\frac{\mathbf{M}[S^\eps\res B_\rho(p)]}{\omega_4 \rho^4} \bigg) &\geq  -  \rho^{-5} \int_{B_\rho(p)} \langle X_p(x), \mathbf{n}_S \rangle\, d\mu_{P^\eps} -  \rho^{-5} \int_{B_\rho(p)} \langle X_p(x), \mathbf{n}_S \rangle\, d\mu_{S_0}\\
 &\geq -  \rho^{-4} (\mathbf{M}[P^\eps \res B_\rho(p)] + \mathbf{M}[S_0])\\
&\geq - (2 \rho^{-3}+ \eps \rho^{-4}) \mathbf{M}[\Sigma] - \rho^{-4} \mathbf{M}[S_0]\, .
\end{split}
\end{equation*}
Integrating this for $1\leq r<R$ from $r$ to $R$ yields
\begin{equation*} 
\begin{split}
 \frac{\mathbf{M}[S^\eps\res B_r(p)]}{\omega_4 r^4} &\leq  \Theta(S^\eps,p,R) + (r^{-2} - R^{-2})\mathbf{M}[\Sigma]
+ \eps \frac{1}{3} (r^{-3} - R^{-3})\mathbf{M}[\Sigma]\\ & \quad + \frac{1}{4} (r^{-4} - R^{-4}) \mathbf{M}[S_0]
 \ .
 \end{split}
\end{equation*}
 Letting $R\ra  \infty$ yields the desired estimate.
\end{proof}

By the uniform local area bound we can thus, up to a subsequence, assume that
$$ S^{\eps_i} \ra S'$$
where $S'$ is locally mass-minimising and satisfies 
$$\partial S' = \Sigma_0\times [0,\infty) - S_0\ .$$
We will define for a general integral current $Q$ its slice at height $t$ by
$$Q_t = \partial (Q \res (M\times (t, \infty)) - (\partial Q)\res (t,\infty)$$
which is compatible with the convention used by Ilmanen in \cite{Ilmanen}.
\begin{lemma} \label{lem:initaluniqueness}
We have $S' = S_0 \times [0,\infty)$. 
 \end{lemma}
\begin{proof}
We consider for $t>0$ the slice $S'_t$ of $S'$ at height $t$ as above.

{\bf Claim 1:} There exists a sequence $t_j \ra \infty$ and $C>0$ such that
$$\mathbf{M}[S'_{t_j}]\leq C .$$
This follows since by the coarea formula, and the locally uniform area estimates
$$ \int_{t}^{t+1} \mathbf{M}[S'_\tau]\, d\tau \leq \mathbf{M}[S'\res \{t\leq z \leq t+1\}] \leq C < \infty $$
independent of $t>0$.\\[1ex]
{\bf Claim 2:} There exists a sequence $t'_l \ra \infty$ such that
$$\mathbf{M}[S'_{t'_l}] \ra \mathbf{M} [S_0]$$
Note that $\partial S'_{t_j} = \partial ((\sigma_{t_j})_{\#}(S_0)) = (\sigma_{t_j})_{\#}(\Sigma_0)$.  By Claim 1 and Lemma \ref{lem:iso} there exists $T_j$ such that 
$\partial T_j = (\sigma_{t_j})_{\#}(S_0) - S'_{t_j}$
and $C'>0$ such that
$$\mathbf{M}(T_j) \leq C'$$
independent of $j$. 
Assume that there exists a $\delta>0$ such that
$$  \mathbf{M}[S'_{t}] \geq \mathbf{M} [S_0] + \delta$$
for all $t>t_0$ for $t_0$ suffciently large. Since both $S'$ and $S_0\times [0,\infty)$ are locally mass minimising we obtain
$$ \mathbf{M}[S'\res \{ 0 \leq z \leq t_j\} ]\leq t_j\, \mathbf{M}[S_0] + C' .$$
But then by the coarea formula
$$ (t_j-t_0)(\mathbf{M}(S_0) +\delta) + t_0 \mathbf{M}(S_0) \leq \int_0^{t_j} \mathbf{M}[S'_{\tau}]\, d\tau \leq \mathbf{M}[S'\res \{ 0 \leq z \leq t_j\} ]\leq t_j\, \mathbf{M}[S_0] + C'\, , $$
which yields a contradiction for $j$ sufficiently large. 

Since by assumption $S_0$ is the unique mass-minimising current spanning $\Sigma_0$ we obtain that
$$ (\sigma_{-t'_l})_\#(S_{t'_l}) \ra S_0$$
in flat norm and in mass. Thus there exists a sequence $\delta_l \ra 0$ such that
$$ \mathbf{M}[S'\res \{ 0 \leq z \leq t_j\} ]\leq t_j\, \mathbf{M}[S_0] + \delta_l \, .$$
But again this yields 
$$ \int_0^{t'_l} \mathbf{M}[S'_{\tau}]\, d\tau \leq \mathbf{M}[S'\res \{ 0 \leq z \leq t_j\} ]\leq t'_l\, \mathbf{M}[S_0] + \delta_l\, , $$
which implies
$$\int_0^{t'_j} \mathbf{M}[S'_{\tau}] - \mathbf{M}[S_0]\, d\tau \leq \delta_l$$
and thus in the limit $l\ra \infty$ that 
$$\mathbf{M}[S'_t] = \mathbf{M}[S_0]$$
and $S'_t = (\sigma_t)_\#(S_0)$ for all $t>0$. 
\end{proof}

\section{Vanishing of the spanning area at the final time}\label{sec: area vanishing final}

For convenience of notation we will in the following replace a sub- or superscript $\eps_i$ by $i$.

Let $l >1$. We choose $\varphi_l \in C^2_c(\R)$ such that $0 \leq \varphi_l \leq 1/l$ with $\varphi_l = 1/l$ on $[2,l+2]$, $\varphi_l = 0$ on $(0,\infty) \setminus [1,l+3]$ and $|D\varphi_l|, |D^2 \varphi_l| \leq 2/l$. 
 
 Recall that $\mu^i_t$ is the associated Radon measure of $P^{\eps_i}(t)$. We denote with $\mu^{S,i}_t$ the associated Radon measure of
 $$ S^i(t) = (\sigma_{-t/{\eps_i}})_\#(S^{\eps_i})\, .$$
 We define the approximate volume by 
 $$ V^i_t:= \int \varphi_l\, d\mu^{S,i}_t\ .$$

We show that the approximate volume goes to zero as $t\rightarrow T_\text{max}$.

\begin{lemma}\label{lem:areasmall}
There exists a constant $C>0$ such that the following holds. Let $0<t<T_\text{max}$ and $l>1$. Then for $i$ sufficiently large 
$$ \mathbf{M}[S^i(t) \res (M\times [2,2+l])] \leq  (1+l) |T_\text{max}-t|^{1/2} \mathbf{M}[\Sigma_0] + C\, .$$
\end{lemma}

\begin{proof}
 We use $T$ to construct a competitor to $S^i(t)$. Recall that
 $$ \partial S^i(t) = P^i(t) + (\sigma_{-t/{\eps_i}})_\#(S_0) $$
 and by Lemma \ref{lem:weakconv} that
\begin{equation}\label{eq:areasmall.1}
  P^i(t) \ra \pi_\#(T_t)\times \R \, .
 \end{equation}
Let $\bar{T} = \pi_\#(T_{[t, \infty)})$. Note further that 
$$ \partial \bar{T} = \partial (\pi_\#(T_{[t, \infty)})) =  \pi_\#(T_t)$$
and since $\spt T \subset M \times [0,T_\text{max}]$ by \eqref{eq:massbound2} 
$$ \mathbf{M}[\pi_\#(T_{[t,\infty )})] \leq (T_\text{max}-t )^{1/2}\mathbf{M}[\Sigma_0]\ .$$
By the co-area fromula (as in the proof of claim 1 in the proof of Lemma \ref{lem:initaluniqueness}), there exists a $C>0$ and $\eta_i \in [0,1]$ such that for all $i$
\begin{equation}\label{eq:areasmall.2}
 \bfM [S^i_{1+\eta_i}] +  \bfM [S^i_{2+l+\eta_i}] \leq C\ .
\end{equation}
Note further that
$$ \pi_\#(P^i(t)\res [1+\eta_i,\infty)) = \pi_\#(T^i_{[t+(1+\eta_i)\eps_i,\infty)}) \ra \bar{T}$$
as well as 
$$ \pi_\#(P^i(t)\res [2+l+\eta_i,\infty)) = \pi_\#(T^i_{[t+(2+l+\eta_i)\eps_i,\infty)}) \ra \bar{T}\, .$$
Consider
$$ R^i_-= (\sigma_{1+\eta_i})_\#(\pi_\#(P^i(t)\res(M \times [1+\eta_i,\infty))) \ra (\bar{T}\times\R)_{1+\eta_i}$$
and
$$ R^i_+= (\sigma_{2+l+\eta_i})_\#(\pi_\#(P^i(t)\res(M\times [2+l+\eta_i,\infty))) \ra (\bar{T}\times\R)_{2+l+\eta_i}$$
and note that
$$ \partial ((S^i(t))_{1+\eta_i}) = \partial R^i_- \ \ \text{and}\ \ \partial ((S^i(t))_{2+l+\eta_i}) = \partial R^i_+ .$$
 By the uniform mass bounds on $S^i_{1+\eta_i},S^i_{2+l+\eta_i},R^i_-, R^i_+$ given by \eqref{eq:areasmall.2}, \eqref{eq:massbound2b} together with  Lemma \ref{lem:maxexist} and Lemma \ref{lem:iso} there exits $D^i_-, D^i_+$ such that
$$ \partial D^i_- = R^i_- - S^i_2 \ \ \text{and}\ \ \partial D^i_+ = R^i_+ - S^i_{2+l}\, ,$$
and a constant $C$ such that
$$ \mathbf{M}[D^i_\pm] \leq C.$$
We can now assume that $\eta_i \rightarrow \eta \in [0,1]$, and thus note
$$ P^i(t)\res (M\times [1+\eta_i,l+2+\eta_i]) + R^i_ - -R^i_+ \ra \partial ((\bar{T}\times \R) \res (M\times [1+\eta,l+2+\eta])) ,$$
and thus by equivalence to convergence in the flat norm, there exits $Q^i$ such that
$$\partial Q^i = ( P^i(t)\res (M\times [1+\eta_,l+2+\eta_i]) + R^i_ - -R^i_+) - \partial ((\bar{T}\times \R) \res (M\times [1+\eta,l+2+\eta]))$$
and $\mathbf{M}[Q_i] \ra 0$. Since $S^i$ is locally mass minimising we can use
$$ ((\bar{T}\times \R) \res (M\times [1+\eta,l+2+\eta])) + Q_i -D^i_-+ D^i_+ $$
as a competitor to get the desired estimate.
\end{proof}

\begin{corollary}\label{cor:limiting volume}
 For every $\varepsilon>0$ there exists $l_0>0, \delta >0$ such that if  $l\geq l_0$, $T_\text{max}-\delta <t< T_\text{max}$, then
 $$V^i_t \leq \varepsilon$$
for $i$ sufficiently large.
\end{corollary}
\begin{proof}
 We can estimate, using Lemma \ref{lem:areasmall}, that for $i$ sufficiently large
 \begin{equation*}
 \begin{split}
 V^i_t= \int \varphi_l\, d\mu^{S,i}_t &\leq 1/l\, \mathbf{M}[S^i(t) \res (M\times [2,2+l])] + 1/l\, \mathbf{M}[S^i(t) \res (M\times [1,2])]\\
 &\ \ \  + 1/l \, \mathbf{M}[S^i(t) \res (M\times [l+2, l+3])] \\
 &\leq 2\delta^{1/2} \mathbf{M}[\Sigma_0] + C/l <  \varepsilon
 \end{split}
 \end{equation*}
\end{proof}

\section{The monotonicity calculation} \label{sec: mon calc}
Recall the approximate volume 
 $$ V^i_t:= \int \varphi_l\, d\mu^{S,i}_t\, ,$$
 where $|\nabla\varphi_l|, |\nabla^2 \varphi_l| \leq 2/l$. Note that we can further assume that 
\begin{equation}\label{eq:moncalc.1}
 \frac{|\nabla \varphi_l|^2}{\varphi_l} \leq C/l\, .
\end{equation}
We define the approximate area as
$$ A^i_t:= \int \varphi_l \, d\mu^i_t\, .$$
We compute
\begin{equation*}\begin{split}
 \frac{d}{dt}V^i_t &= \frac{d}{dt} \int \varphi_l  d\mu^{S,i}_t =  \int \langle \nabla\varphi_l, - \eps^{-1}\tau\rangle + \varphi_l\, \text{div}( - \eps^{-1}\tau)\, d\mu^{S,i}_t\\
  &= - \int \langle \nabla \varphi_l, \eps^{-1}\tau^\perp \rangle + \text{div}(\varphi_l \eps^{-1}\tau)\, d\mu^{S,i}_t \\
 &= - \int \langle \nabla \varphi_l, \eps^{-1}\tau^\perp \rangle\, d\mu^{S,i}_t- \int \varphi_l \, \langle \eps^{-1}\tau, \mathbf{n}\rangle\, d\mu^i_t \\
 &= - \int \langle \nabla\varphi_l, \eps^{-1}\tau^\perp \rangle \, d\mu^{S,i}_t + \int \varphi_l\, \langle \mathbf{H}, \mathbf{n}\rangle\, d\mu^i_t
  \end{split}
 \end{equation*}
where we used Lemma \ref{lem:div thm} in the step from the second to the third line. Following an idea of Huisken-Ilmanen \cite[Lemma 5.3]{HuiskenIlmanen}, we can rewrite the first term as a derivative:
\begin{equation*}\begin{split}
 \int \langle \nabla \varphi_l, \eps^{-1}\tau^\perp \rangle \, d\mu^{S,i}_t &=  \int \eps^{-1} \varphi_l'(z) \langle \tau, \tau^\perp \rangle \, d\mu^{S,i}_t\\
 &= \int \eps^{-1} \varphi'_l(z-\eps^{-1}t) \langle \tau, \tau^\perp \rangle \, d\mu^{S,i}_0 \\
 &=  - \frac{d}{dt} \int \varphi_l(z-\eps^{-1}t) \langle \tau, \tau^\perp \rangle \, d\mu^{S,i}_0\\
 &= - \frac{d}{dt} \int \varphi_l(z) \langle \tau, \tau^\perp \rangle \, d\mu^{S,i}_t\, .
  \end{split}
  \end{equation*}

Thus we get for $0\leq t_1 < t_2$ that
\begin{equation}\label{eq:moncalc.2}
\begin{split}
 V^i_{t_2}-V^i_{t_1} & = \int_{t_1}^{t_2} \int \varphi_l\, \langle \mathbf{H}, \mathbf{n}\rangle\, d\mu^i_t\, dt \\
 &\ \ \ + \int \varphi_l\, \langle \tau, \tau^\perp \rangle \, d\mu^{S,i}_{t_2} - \int \varphi_l\, \langle \tau, \tau^\perp \rangle \, d\mu^{S,i}_{t_1} 
\end{split}
 \end{equation}

For the approximate area we get, using \eqref{eq:moncalc.1},
\begin{equation*}
\begin{split}
\frac{d}{dt} A^i_t = \int \langle \nabla \varphi_l, \mathbf{H}\rangle -\varphi_l |\mathbf{H}|^2\, d\mu^i_t &\leq \int_{\{\nabla\varphi_l \neq 0\}} \frac{|\nabla \varphi_l|^2}{\varphi_l} d\mu^i_t - \frac{1}{2}\int \varphi_l |\mathbf{H}|^2\, d\mu^i_t\\
&\leq \frac{C}{l} - \frac{1}{2}\int \varphi_l |\mathbf{H}|^2\, d\mu^i_t
\end{split}
\end{equation*}
where we used the uniform local area bounds to estimate the first integral on the right hand side. This implies the estimate
\begin{equation}\label{eq:moncalc.3}
A^i_{t_2} + \frac{1}{2} \int_{t_1}^{t_2}\int\varphi_l |\mathbf{H}|^2\, d\mu^i_t\, dt \leq A^i_{t_1} + \frac{C}{l} (t_2-t_1)
\end{equation}
Note that by \eqref{eq:massbound1} this implies that 
\begin{equation}\label{eq:moncalc.4}
A^i_t \leq 2 \mathbf{M}[\Sigma_0] + \frac{C}{l} t\ .
\end{equation}
We define the function $f_\kappa:\R^+\ra \R^+$ by
\begin{equation*}
f_\kappa(A):=\int_0^A\frac{a^\frac{1}{2}}{(16\pi+ 4\kappa a)^\frac{1}{2}}\,
  da \ .
 \end{equation*}
Note that $f_0 = \tfrac{1}{6 \sqrt{\pi}} a^{3/2}$.

Let $(M_t)_{0\leq t <T}$ be a smooth mean curvature flow of closed, embedded hypersurfaces in a Cartan-Hadamard manifold $(M^3,g)$ with sectional curvatures bounded above by $-\kappa$. Let $V(t)$ be the enclosed volume and $A(t)$ the area. We then can apply Theorem \ref{lem:lowerbound}
to estimate 
\begin{equation*}
\begin{split}
-\frac{d}{dt}V =&\ \int_{M_t} H\, d\ch^2 \leq \bigg( \int_{M_t}
H^{2}\, d\ch^2\bigg)^\frac{1}{2} A^\frac{1}{2}\\
&\ \cdot (16\pi+4\kappa A)^{-\frac{1}{2}}\bigg( \int_{M_t}
H^{2}\, d\ch^2\bigg)^\frac{1}{2} \\
=&\ (16\pi+4\kappa A)^{-\frac{1}{2}}A^\frac{1}{2}\int_{M_t}
H^{2}\, d\ch^2 = -\frac{d}{dt}f_\kappa(A)\ .
\end{split} 
\end{equation*}
 Thus $f_\kappa(A)-V$ is monotonically decreasing under the flow. Consider the case
that $M_\kappa^3$ is the model space of constant curvature
$-\kappa$ and let $M_t$ be the mean curvature flow of geodesic spheres
contracting to a point. Then the estimate of Theorem \ref{lem:lowerbound} holds with equality for all
$M_t$ and also the above calculation is an equality. Using that in the
model space geodesic balls optimize the isoperimetric ratio, we have 
 $$\mathcal{H}^3(U) \leq f_\kappa(\ch^2(\partial U))\ ,$$ 
for all open and bounded $U\subset M^3_\kappa$, with
equality on geodesic balls.

We consider the approximate isoperimetric difference
$$
I^i_t = f_\kappa(A^i_t) - V^i_t\ ,$$
and compute
\begin{equation*}
\begin{split}
f_\kappa(A^i_{t_2}) - f_\kappa(A^i_{t_1}) = &- \int_{t_1}^{t_2} \frac{(A^i_t)^{1/2}}{(16\pi + 4\kappa A^i_t)^\frac{1}{2}}  \int \varphi_l |\mathbf{H}|^2 \, d\mu^i_t\, dt\\
& +\int_{t_1}^{t_2}  \frac{(A^i_t)^{1/2}}{(16\pi + 4\kappa A^i_t)^\frac{1}{2}} \int \langle \nabla\varphi_l, \mathbf{H}\rangle \, d\mu^i_t\, dt
\end{split}
\end{equation*}
where we can estimate, assuming $t_2 \leq T_\text{max}$ and using \eqref{eq:moncalc.3}, \eqref{eq:moncalc.4}
\begin{equation*}\begin{split}
 \left| \int_{t_1}^{t_2} \frac{(A^i_t)^{1/2}}{(16\pi + 4\kappa A^i_t)^\frac{1}{2}} \int \langle \nabla \varphi_l, \mathbf{H}\rangle \, d\mu^i_t\, dt \right| &\leq 
 \left| \int_{t_1}^{t_2} \frac{(A^i_t)^{1/2}}{(16\pi + 4\kappa A^i_t)^\frac{1}{2}} \int |\nabla \varphi_l| |\mathbf{H}| \, d\mu^i_t\, dt \right| \\
 &\leq  C \left( \int_{t_1}^{t_2} \int_{\{\nabla \varphi_l \neq 0\}} \frac{|D\varphi|^2}{\varphi} \, d\mu^i_t\, dt \right)^{1/2}\\
 &\qquad \left( \int_{t_1}^{t_2} \int \varphi_l |\mathbf{H}|^2 \, d\mu^i_t\, dt \right)^{1/2}\\
 &\leq \frac{C}{l^{1/2}}(t_2-t_1)\, .
\end{split}
\end{equation*}
This yields the estimate
\begin{equation}\label{eq:moncalc.10}
\begin{split}
f_\kappa(A^i_{t_2}) - f_\kappa(A^i_{t_1}) &\leq - \int_{t_1}^{t_2} \frac{(A^i_t)^{1/2}}{(16\pi + 4\kappa A^i_t)^\frac{1}{2}} \int \varphi_l |\mathbf{H}|^2 \, d\mu^i_t\, dt \\
&\qquad + \frac{C}{l^{1/2}}(t_2-t_1)
\end{split}
\end{equation}
for $0\leq t_1 <t_2 \leq T_\text{max}$.

From \eqref{eq:moncalc.2}, \eqref{eq:moncalc.10} we get the estimate 
\begin{equation}\label{eq:moncalc.11}
\begin{split}
 I^i_{t_2}-I^i_{t_1} & \leq  - \int_{t_1}^{t_2} \frac{(A^i_t)^{1/2}}{(16\pi + 4\kappa A^i_t)^\frac{1}{2}} \int \varphi_l |\mathbf{H}|^2 \, d\mu^i_t\, dt  +  \int_{t_1}^{t_2} \int \varphi_l |\mathbf{H}|\, d\mu^i_t\, dt \\
 &\ \ \ + \int \varphi_l\, \langle \tau, \tau^\perp \rangle \, d\mu^{S,i}_{t_1} - \int \varphi_l\, \langle \tau, \tau^\perp \rangle \, d\mu^{S,i}_{t_2}\\
 &\ \ \ + \frac{C}{l^{1/2}}(t_2-t_1)\\
 &\leq  - \int_{t_1}^{t_2} \frac{(A^i_t)^{1/2}}{(16\pi + 4\kappa A^i_t)^\frac{1}{2}} \int \varphi_l |\mathbf{H}|^2 \, d\mu^i_t\, dt\\
 &\ \ \  +  \int_{t_1}^{t_2} \left(\int \varphi_l |\mathbf{H}|^2\, d\mu^i_t\right)^{1/2} (A^i_t)^{1/2}\, dt \\
 &\ \ \ + \int \varphi_l\, \langle \tau, \tau^\perp \rangle \, d\mu^{S,i}_{t_1} - \int \varphi_l\, \langle \tau, \tau^\perp \rangle \, d\mu^{S,i}_{t_2}\\
 &\ \ \ + \frac{C}{l^{1/2}}(t_2-t_1)
 \end{split}
 \end{equation}
for $0\leq t_1 <t_2 \leq T_\text{max}$.

Recall that the limiting Brakke flow $(\bar{\mu}_t)_{(0\leq t\leq T_\text{max})}$ is invariant in the $z$-direction, and for a.e.~$t$ the measure $\bar{\mu}_t$ is $3$-rectifiable and carries a weak mean curvature in $L^2$.  Using Theorem \ref{lem:lowerbound} we thus see that
\begin{equation}\label{eq:moncalc.12}
\int\varphi_l |\mathbf{H}|^2 \, d\bar{\mu}_t \geq 16\pi + 4 \kappa\, \bar{\mu}_t(\varphi_l)
\end{equation}
for a.e.~$t \in [0,T_\text{max}]$

\begin{lemma}\label{lem:moncalc.1} For any $t_1,t_2 \in [0,T_\text{max})$, $t_1<t_2$,  
\begin{equation*}
\limsup_{i\ra \infty} \int_{t_1}^{t_2}\!L^i_t\, dt\leq \int_{t_1}^{t_2}\!L_t\, dt \leq 0\ ,
\end{equation*}
where
\begin{equation*}
L^i_t:= \big(A^i_t\big)^\frac{1}{2}\bigg(\bigg(\int \varphi_l
|\mathbf{H}|^2\,
d\mu^i_t\bigg)^\frac{1}{2} -  (16\pi + 4\kappa A^i_t)^{-\frac{1}{2}} \int
\varphi_l |\mathbf{H}|^{2}\, d\mu^i_t\bigg) 
\end{equation*}
and

\begin{equation*}
L_t:=\big(\bar{\mu}_t(\varphi_l)\big)^\frac{1}{2}\bigg(  \bigg(\int \varphi_l
|\mathbf{H}|^2\,
d\bar{\mu}_t\bigg)^\frac{1}{2} - (16\pi + 4\kappa \bar{\mu}_t(\varphi_l))^{-\frac{1}{2}} \int
\varphi_l |\mathbf{H}|^{2}\, d\bar{\mu}_t\bigg)\ .
\end{equation*}

\end{lemma}
\begin{proof}
From Ilmanen's compactness theorem for Brakke flows, we know that for all $t \in [t_1,t_2]$ we have
$$ A^i_t \rightarrow \bar{\mu}_t(\varphi_l)\ ,$$
and by the lower semicontinuity of the $L^2$ norm of $\mathbf{H}$ and \eqref{eq:moncalc.12} that
\begin{equation}\label{eq:moncalc.6.2}
 \liminf_{i \rightarrow \infty} \int
\varphi_l |\mathbf{H}|^2\, d\mu^i_t \geq \int
\varphi_l |\mathbf{H}|^2\, d\bar{\mu}_t \geq 16\pi + 4 \kappa\, \bar{\mu}_t(\varphi_l)\ .
\end{equation}

We write $L^i_t$ in the form $L^i_t= a_i \cdot b_i$, where
\begin{equation*}
 a_i(t):=\big(A^i_t\big)^\frac{1}{2}
\end{equation*}
and
\begin{equation*}
b_i(t):=  \bigg(\int \varphi_l
|\mathbf{H}|^2\,
d\mu^i_t\bigg)^\frac{1}{2} - (16\pi + 4\kappa A^i_t)^{-\frac{1}{2}} \int
\varphi_l |\mathbf{H}|^2\, d\mu^i_t
\end{equation*}
Since the function $x^{1/2} - (16 \pi + 4\kappa a)^{-\frac{1}{2}} x$ is decreasing for $x\geq 4\pi + \kappa a$ we obtain, using \eqref{eq:moncalc.6.2}, that
$$ \limsup_{i \rightarrow \infty} b_i(t) \leq  \bigg(\int \varphi_l
|\mathbf{H}|^2\,
d\bar{\mu}_t\bigg)^\frac{1}{2} - (16\pi + 4 \kappa \bar{\mu}_t(\varphi_l) )^{-\frac{1}{2}} \int
\varphi_l |\mathbf{H}|^2\, d\bar{\mu}_t \leq 0  $$ 
for all $t \in [t_1,t_2]$. Together with $a_i(t)\rightarrow \big(\bar{\mu}_t(\varphi_l)\big)^{1/2}$ this implies that
$$ \limsup L^i_t \leq L_t \leq 0 \qquad \text{for all } t \in [t_1,t_2]\, .$$
Note that by \eqref{eq:moncalc.4} there is $C\geq 0$ such that
\begin{equation*}
L^i_t \leq C
\end{equation*}
for all $t\in [t_1,t_2]$ and all $i\geq i_0$. Then the claim follows from Fatou's lemma.
\end{proof}

We will now explain how to perturb $\Sigma$ slightly such that we can assume that the mass minimising 3-current spanning $\Sigma$ is unique. Let $S$ be any mass minimising 3-current spanning $\Sigma$. Note that by
Almgren \cite{Almgren2000}, see also De~Lellis-Spadaro \cite{DeLellisSpadaro11,  DeLellisSpadaro15, DeLellisSpadaro14, DeLellisSpadaro16a, DeLellisSpadaro16b}, the interior singular set of $S$ has codimension $2$. Note further that the interior regular set $\mathcal{R}_\text{int}$ can have at most countably connected components, since $S$ has finite mass. We denote these components by $R_j$ for $j \in 1,\ldots, N$ where $N \in \N \cup \{\infty\}$. We can pick points $p_j \in R_j$ and radii $r_j> 0$ such that 
\begin{itemize}
\item[-] $B_{r_j}(p_j) \cap \spt \Sigma = \emptyset$, 
\item[-] the balls $B_{r_j}(p_j)$ are pairwise disjoint,
\item[-] $S \res B_{r_j}(p_j)$ is smooth and consists of one single, smooth, connected component,
\item[-] $\sum_{i=1}^N \mathbf{M}[S\res B_{r_j}(p_j)] \leq 1$.
\end{itemize}
For $k \in \N$ let 
$S_k = S \setminus \cup_{j=1}^N B_{r_j/k}(p_j)$ and 
$$ \Sigma_k := \partial S_k\ .$$
The above estimates yield that
\begin{equation}\label{eq:moncalc.6.3}
\Sigma_k \rightarrow \Sigma
\end{equation}
in flat norm and in mass.  

\begin{lemma}\label{lem:moncalc.2}
$S_k$ is the unique mass minimizing current spanning $\Sigma_k$.
\end{lemma}
\begin{proof}  Assume there is another mass minimizing current $S'$ spanning $\Sigma_k$. But then
$$ S'':= S' + \cup_{j=1}^N S\res \overline{B}_{r_j/k}(p_j) $$
is mass minimising and bounds $\Sigma$. The interior singular set of $S''$ has again codimension $2$. Thus by unique continuation, see \S \ref{sec:unique continuation}, $S''$ has to coincide with $S$ on each connected component $R_j$ of $\mathcal{R}_\text{int}$. Note that $S''$ can have no further connected components of its interior regular set, since it has the same mass as $S$.  Thus $S'' = S$ and $S'=S_k$. \end{proof}

\begin{proof}[Proof of Theorem \ref{thm:mainthm}] We will present the proof in the case $\Sigma$ is an integral $2$-current. The necessary modifications if $\Sigma$ is a 2-dimensional flat chain mod 2 will be discussed at the end of the proof.

We first replace $\Sigma$ by $\Sigma_k$ such that by Lemma \ref{lem:moncalc.2} $S_k$ is the unique area minimising current spanned by $\Sigma_k$. 

As outlined above we use Ilmanen's elliptic regularisation scheme to construct Brakke flows $(\mu^k_t)_{(0\leq t\leq T^k_\text{max})}$ , starting at $\Sigma_{k,0}:= \Sigma_k$, which vanish at a finite time $T^k_\text{max}$. These flows arise as the slice of the translation invariant flows $(\bar{\mu}^k_t)_{(0\leq t\leq T^k_\text{max})}$ on $M \times \R$, obtained as a limit of approximating flows $(\mu^{k,i}_t)_{t \geq 0}$. 

We will for the moment omit the index $k$. We use the set-up as before. Note that by Lemma \ref{lem:initaluniqueness} we have that
\begin{equation}\label{eq:moncalc.7}
S^i(0) = S^i \ra S \times [0,\infty)\ .
 \end{equation}

Given $\eps>0$, we choose $t_1=0$ and $T_\text{max} -\delta < t_2 < T_\text{max}$, $l >l_0$, where $\delta>0$ and $l_0$ are given by Corollary \ref{cor:limiting volume}. By \eqref{eq:moncalc.11} and Lemma \ref{lem:moncalc.1} we can estimate 
\begin{equation}\label{eq:moncalc.8}
\begin{split}
\limsup_{i\ra \infty} (I^i_{t_2}-I^i_{0}) & \leq   \limsup_{i\ra \infty} \left|\int \varphi_l\, \langle \tau, \tau^\perp \rangle \, d\mu^{S,i}_{0}\right |
  \\
  & \quad+ \limsup_{i\ra \infty}\left| \int \varphi_l\, \langle \tau, \tau^\perp \rangle \, d\mu^{S,i}_{t_2}\right|
   + \frac{C}{l^{1/2}}T_\text{max}
 \end{split}
 \end{equation}
 Note that by \eqref{eq:moncalc.7} we have that
 $$ \lim_{i\ra \infty} \int \varphi_l\, \langle \tau, \tau^\perp \rangle \, d\mu^{S,i}_{0} = 0 $$
 and by Corollary \ref{cor:limiting volume}
 $$ \limsup_{i\ra \infty}\left| \int \varphi_l\, \langle \tau, \tau^\perp \rangle \, d\mu^{S,i}_{t_2}\right| \leq \varepsilon.$$
 Furthermore
 $$\lim_{i\ra \infty} I^i_{0} = f_\kappa \left( \int_{\Sigma_k \times \R} \varphi_l \, d\ch^3\right) - \int_{S_k \times \R} \varphi_l \, d\ch^4 \ .$$
 Again by Corollary \ref{cor:limiting volume} we have 
 $$\limsup_{i\ra \infty} I^i_{t_2} \geq -  \varepsilon\ .$$
 Putting this together we obtain
 \begin{equation*}
f_\kappa \left( \int_{\Sigma_k \times \R} \varphi_l \, d\ch^3\right) -  \int_{S_k \times \R} \varphi_l \, d\ch^4 \geq - C \varepsilon - \frac{C}{l^{1/2}}\ .
\end{equation*}
 Letting first $l \rightarrow \infty$ and then $\eps \ra 0$ yields the isoperimetric inequality for $\Sigma_k$. We can now let $k \rightarrow \infty$ to obtain the isoperimetric inequality for $\Sigma$.
 
{\bf The equality case:} In case of equality in \eqref{eq:main.1} we have
$$ f_\kappa\left(\mathbf{M}[\Sigma_k]\right) -  \mathbf{M}[S_k] \rightarrow 0$$
and the monoticity calculation, using Lemma \ref{lem:moncalc.1}, yields
$$ \int_0^{T_{\text{max},k}}\!\!\! \big(\mu_{k,t}(M)\big)^\frac{1}{2}\bigg( \bigg(\int
|\mathbf{H}|^2\,
d\mu^k_{t}\bigg)^\frac{1}{2} - (16\pi + 4 \kappa \mu^k_{t} (M) )^{-\frac{1}{2}}
 \int
 |\mathbf{H}|^{2}\, d\mu^k_{t}\bigg)\, dt \rightarrow 0 $$
where $(\mu^k_{t})_{t\geq 0}$ is the constructed Brakke flow starting at $\Sigma_k$. We aim to let $k \ra \infty$ and construct a non-vanishing Brakke flow starting at $\Sigma$. 

By Lemma \ref{lem:minexist}, there exists $\delta >0$ and $\eta>0$ such that
\begin{equation}\label{eq:moncalc.8b}
 \mu^k_{t}(M) \geq \eta 
 \end{equation}
for all $0\leq t < \delta$ and all $k$ sufficiently large.

We can thus consider a subsequential limit as $k \ra \infty$ and obtain a limiting Brakke flow $\{\mu_t\}_{t\geq 0}$ which satisfies \eqref{eq:moncalc.8b} as well. Similarly as in the proof of Lemma \ref{lem:moncalc.1} we obtain

$$ \int_0^\delta \big(\mu_t(M)\big)^\frac{1}{2}\bigg(  \bigg(\int
|\mathbf{H}|^2\,
d\mu_{t}\bigg)^\frac{1}{2} -(16\pi + 4 \kappa \mu_{t} (M) )^{-\frac{1}{2}} \int
 |\mathbf{H}|^{2}\, d\mu_{t}\bigg)\, dt = 0\ .$$
Thus for a.e.~$t \in (0,\delta)$, using \eqref{eq:moncalc.8b}, we have
\begin{equation}\label{eq:moncalc.9}
\int
|\mathbf{H}|^2\,
d\mu_{t} = 16 \pi + 4\kappa \mu_t(M) \ .
\end{equation}
Thus by Theorem \ref{lem:lowerbound} for a.e.~$t \in (0,\delta)$, $\mu_t$ is the Radon measure associated to
 a smooth embedded $2$-sphere $\Sigma_t$ with density one, where the mean curvature vector has constant length and $\Sigma_t$ is totally umbilic. Furthermore, it bounds a totally geodesic embedded $3$-ball $S_t$, with the mean curvature vector of $\Sigma_t$ proportional to the unit conormal of $S_t$ at every point in $\Sigma_t$.  $S_t$ is isometric to a geodesic ball in the 3-dimensional model space such that the mean curvature of the boundary coincides with the one of $\Sigma_t$. Since all the flows $\{\mu^k_t\}_{t\geq 0}$ are unit regular, see \cite[\S 4]{mcftripleedge}, White's local regularity theorem \cite{White05}, implies that the convergence is smooth for $0<t<\delta$ and the limiting flow $\{\mu_t\}_{0<t<\delta}$ is smooth (and thus the above characterisation of $\mu_t$ holds for all $t \in (0,\delta)$).

The smooth convergence implies that $\mu^k_t = \mu_{T^k_t}$ for any $0<t<\delta$ and $k$ sufficiently large. Recall that
in the flat norm
$$ \text{dist}(\Sigma_k,T^k_t) \leq (t + t^{1/2}) \mathbf{M}[\Sigma_k] $$
and $\Sigma_k \ra \Sigma$ in flat norm. This yields that 
$$\Sigma_t \ra \Sigma$$
 in flat norm. Since $\Sigma_t$ converges smoothly to a limit as $t \searrow 0$ this yields the claimed statement about $\Sigma$.

In the case that $\Sigma$ is a 2-dimensional flat chain mod 2, we work with flat chains mod 2 instead of integral currents. Note that Ilmanen's elliptic regularisation scheme works analogously in this setting. All the other parts of the argument also directly carry over. The only point to note is that the interior regularity of Almgren \cite{Almgren2000} has to be replaced by the corresponding result for flat chains mod 2 due to Federer \cite{Federer70}.
\end{proof}

\section{An optimal lower bound on the Willmore energy} \label{sec: lower bound Willmore}
In this section we give the proof of the optimal lower bound on the Willmore energy.\\[-2ex]
\begin{proof}[Proof of Theorem \ref {lem:lowerbound}] We consider the vectorfield $X$ given by
\begin{equation*}
X:= \phi(r) \bar{\nabla} r\ ,
\end{equation*}
where $r(p):=\text{dist}_{M}(p,p_0)$ for a fixed
$p_0\in N$ and $\varphi \in C^{0,1}_\text{loc}[0,\infty), \varphi\geq 0$. Here $\bar{\nabla}$ denotes the gradient
operator on $M$. The distance function to a point on such a manifold is smooth away from $p_0$ and satisfies, see for example \cite{Petersen}:
\begin{equation*}
\begin{split}
\bar{\nabla}r \neq&\ 0 \\
\text{Hess}(r)=\bar{\nabla}^2r\geq&\
\Psi(r)\big(\text{id}-\bar{\nabla}r\otimes\bar{\nabla}r\big) ,
\end{split}
\end{equation*}
for $p\neq p_0$, where $\Psi(r)=1/r$ for $\kappa = 0$ and
$\Psi(r)=\sqrt{\kappa}\cosh(\sqrt{\kappa}r)/\sinh(\sqrt{\kappa}r)$
for $\kappa >0$ and the second inequality holds w.r.t.\! an orthonormal basis of $T_pM$. For a
point $p\in \Sigma$, such that the tangent space of $\Sigma$ exists at $p$ we compute
\begin{equation}\label{eq:intest.1}
\begin{split}
\text{div}_{\Sigma}(X) =&\ \text{div}_\Sigma (\varphi \bar{\nabla}r) =
\varphi \,\text{div}_\Sigma(\bar{\nabla}r) + \varphi' \bar{g}(\nabla_\Sigma r, \bar{\nabla}r)\\
=&\  \varphi\, \text{tr}_{T_p\Sigma}\big(\text{Hess}(r)\big) + \varphi'\big (1 - |(\bar{\nabla} r)^\perp|^2\big)\\
\geq&\ \varphi\,\Psi \,
\text{tr}_{T_p\Sigma}\big(\text{id}-\bar{\nabla}r\otimes\bar{\nabla}r\big)
+ \varphi'\big (1 - |(\bar{\nabla} r)^\perp|^2\big)\\
=&\ \varphi\,\Psi \big(1+  |(\bar{\nabla} r)^\perp|^2\big) +  \varphi'\big (1 - |(\bar{\nabla} r)^\perp|^2\big)\\
=&\ \varphi\, \Psi + \varphi' + \big(\varphi\, \Psi - \varphi') |(\bar{\nabla} r)^\perp|^2 \ .
\end{split}
\end{equation}
Pick any $p_0\in \Sigma$ such that the density $\Theta(p_0)$ exists and
$\Theta(p_0)\geq 1$. \\[1ex]
{\bf The case $\kappa =0$:} Given $\sigma>0$ we choose
$$\varphi(r) = \frac{r}{r_\sigma^2}$$
where  $r_\sigma=\max (r,\sigma)$. This gives
\begin{equation}\label{eq:intest.2}
\int \text{div}_{\Sigma}(X)\, d\mu \geq 2\sigma^{-2} \mu (B_\sigma(p_0)) + \int_{M \setminus B_\sigma(p_0)} 
2 |X^\perp|^2\, d\mu\, .
\end{equation}
On the other hand, applying the divergence theorem yields
\begin{equation}\label{eq:intest.3}
\int \text{div}_{\Sigma}(X)\, d\mu =  - \sigma^{-2} \int_{B_\sigma(p_0)} r\, \bar{g}(\bar{\nabla}r, \bfH)\, d\mu - \int_{ M \setminus B_\sigma(p_0)} \bar{g}(X, \bfH)\, d\mu\ .
\end{equation}
Combining both equations yields
\begin{equation*}\begin{split}
2\sigma^{-2}\mu(B_\sigma(p_0))+ 2\int_{M\setminus
  B_\sigma(p_0)} \bigg|\frac{1}{4}\mathbf{H}+ X^\perp\bigg|^2
d\mu \leq \ &\frac{1}{8}\int_{M\setminus
  B_\sigma(p_0)}|\mathbf{H}|^2\, d\mu\\  &-\sigma^{-2}\int_{B_\sigma(p_0)}r \bar{g}(\bar{\nabla}r,\mathbf{H})\, d\mu\ .
\end{split}
\end{equation*}
Since 
\begin{equation}\label{eq:intest.4}
\lim_{\sigma\ra 0}\sigma^{-2}\mu(B_\sigma)(p_0) \geq \pi\ ,
\end{equation}
 we can take the limit $\sigma \ra 0$ to obtain
 \begin{equation}\label{eq:intest.5}
2\pi+ 2\int \bigg|\frac{1}{4}\mathbf{H}(x)+ \frac{\big(\bar{\nabla} r_{p_0}(x)\big)^\perp}{r_{p_0}(x)}\bigg|^2
d\mu (x) \leq \frac{1}{8}\int |\mathbf{H}|^2\, d\mu
\end{equation}
for any $p_0$ such that \eqref{eq:intest.4} holds.\\[1ex]
{\bf  The case $\kappa >0$:} we  can assume w.l.o.g.~via rescaling that $\kappa =1$. We choose
$$\varphi_\sigma(r) = \frac{\sinh(r)}{(2\cosh(r) -2)_{\sigma^2}}$$
where $(2\cosh(r) -2)_{\sigma^2} = \max(2\cosh(r) -2, \sigma^2)$. The choice of $\varphi_\sigma$ will become clear in the discussion of the equality case. We further denote $\sigma'= \sigma'(\sigma)$ to be the solution of
$$ 2\cosh(\sigma') -2 = \sigma^2 $$
Note that, 
\begin{equation}\label{eq:intest.6}
\lim_{\sigma \ra 0} \frac{\sigma'}{\sigma} = 1\ .
\end{equation}
Using $\psi(r)= \cosh(r)/\sinh(r)$, for $r<\sigma$ we have, suppressing the index $\sigma$,
\begin{equation*}\label{eq:intest.7}
 \phi \,\psi+ \phi' = 2 \sigma^{-2}\cosh(r) \geq 2 \sigma^{-2} \qquad \text{and} \qquad \phi \, \psi - \phi' = 0\ ,
\end{equation*}
and for $r\geq \sigma$, noting that
\begin{equation*}
 \phi'(r) = -\frac{1}{2\cosh(r) -2}
\end{equation*}
we obtain
\begin{equation*}\label{eq:intest.8}
 \phi\, \psi+ \phi' = \frac{1}{2}\qquad \text{and} \qquad \phi\, \psi - \phi' = 2\phi^2\ .
\end{equation*}
Inserting this into \eqref{eq:intest.1} gives
\begin{equation*}\label{eq:intest.9}
\int \text{div}_{\Sigma}(X)\, d\mu \geq 2\sigma^{-2} \mu (B_{\sigma'}(p_0)) + \int_{M \setminus B_{\sigma'}(p_0)} 
\frac{1}{2}+2 |X^\perp|^2\, d\mu \ .
\end{equation*}
Arguing as before, using \eqref{eq:intest.6}, we arrive at
 \begin{equation}\label{eq:intest.10}
2\pi+ 2\int \bigg|\frac{1}{4}\mathbf{H}(x)+ \varphi_0(r_{p_0}(x)) \big(\bar{\nabla} r_{p_0}(x)\big)^\perp \bigg|^2
d\mu (x)  + \int \frac{\kappa}{2}\, d\mu \leq \frac{1}{8} \int |\mathbf{H}|^2\, d\mu
\end{equation}
for any $p_0$ such that \eqref{eq:intest.4} holds.\\[1ex]
{\bf The equality case for $\kappa =0$:} To see that $\Sigma$ is a smoothly embedded $2$-sphere with unit density, we can nearly verbatim follow the argument in \cite[Proposition 2.1]{LammSchaetzle14}. We include it for completeness. We first note that by equality in \eqref{eq:intest.5}, since $p_0 \in \spt \mu$ is arbitrary, we have that $\Sigma$ has unit multiplicity:
\begin{equation}\label{eq:intest.11}
\theta^2(\mu)=1\ \text{on } \spt \mu\ .
\end{equation}
Furthermore \eqref{eq:intest.5} gives that
$$ \mathbf{H}(y)+ 4 \frac{\big(\bar{\nabla} r_{x}(y)\big)^{\perp_y}}{r_{x}(y)} \quad \text{for } \mu\text{-almost all } y\in \spt \mu\, ,$$
where ${}^{\perp_y}$ denotes the orthogonal projection onto $T^\perp_y\mu$. In particular
\begin{equation}\label{eq:intest.12}
 \mathbf{H}(y) \perp T_y\mu\quad \text{for } \mu\text{-almost all } y\in \spt \mu\, .
 \end{equation}
By Fubini's Theorem, for $\mu$-almost all $y$ it holds that
$$\mathbf{H}(y)+ 4 \frac{\big(\bar{\nabla} r_{x}(y)\big)^{\perp_y}}{r_{x}(y)} \quad \text{for } \mu\text{-almost all } x\in \spt \mu\, .$$
We choose any $y \in \spt \mu$ such that $T_y\mu$ exists. If $\mathbf{H}(y) = 0$, then $\spt \mu \subset \exp_y(T_y\mu)$. As in \cite[Proposition 2.1]{LammSchaetzle14} this contradicts the compactness of $\spt \mu$. Hence $\mathbf{H}(y) \neq 0$ and we may assume that $\mathbf{H}(y) \perp T_y\mu$ by \eqref{eq:intest.12}. By scaling and chosing exponential coordinates $x = \exp_y$, we may assume that $T_0\mu = \text{span}\{e_1,e_2\}, T\perp_0\mu = \text{span}\{e_3, \ldots ,e_n\}, \mathbf{H}(0)= 2e_3$ and we write $\perp$ for the projection to $\text{span}\{e_3, \ldots ,e_n\}$ in these coordinates. Denoting with $\langle \cdot , \cdot \rangle$ the metric on $T_yM$, we firstly get from the above for $j=4, \ldots, n$, that
\begin{equation*}\begin{split}
0 &= \left\langle \mathbf{H}(0), e_j \right \rangle = -4 \left\langle \frac{- x^\perp}{|x|^2}, e_j \right \rangle\\
&= 4 \left\langle \frac{x^\perp}{|x|^2}, e_j \right \rangle = 4\frac{x_j}{|x|^2} \quad \text{for } \mu\text{-almost all } x\neq 0 \in \exp^{-1}_y(\spt \mu)\ .
\end{split}\end{equation*}
Thus $\exp^{-1}_y(\spt \mu) \subset \text{span}\{e_1,e_2,e_3\}$. For $j=3$ we obtain
$$ 2 = \left\langle \mathbf{H}(0), e_3 \right \rangle = -4 \left\langle \frac{- x^\perp}{|x|^2}, e_3 \right \rangle = 4\frac{x_3}{|x|^2} \quad \text{for } \mu\text{-almost all } x\neq 0 \in \exp^{-1}_y(\spt \mu)\, .
$$
Thus $2x_3 = |x|^2$ and again as in \cite[Proposition 2.1]{LammSchaetzle14} one sees that 
$$ \mu = \mathcal{H}^2 \res \Sigma$$
where 
$$\Sigma = \exp_y\big(\partial B_1(e_3) \cap \text{span}\{e_1,e_2, e_3\}\big)\ .$$

\smallskip

To construct the spanning flat $3$-ball we argue as follows. Note first that we can repeat the same argument for every point $y \in \Sigma$. Pick $y_0 \in \Sigma$ such that
\begin{equation}\label{eq:intest.13}
|\mathbf{H}(y_0)| = \max_{\Sigma} |\mathbf{H}|
\end{equation}
and denote $r_0 = 2/|\mathbf{H}(y_0)|$. Applying the above argument at $y_0$, but without scaling, we obtain
$$\Sigma = \exp_y\big(\partial B_{r_0}(r_0 e_3) \cap \text{span}\{e_1,e_2, e_3\}\big) .$$
We define
$$S =  \exp_y\big( \overline{B}_{r_0}(r_0 e_3) \cap \text{span}\{e_1,e_2, e_3\}\big) .$$
{\bf Claim:}  $S$ with its induced metric $\tilde{g}$ is isometric via the exponential map at $y_0$ to $B_{r_0}(r_0 e_3) \subset \R^3$.\\[1ex]
Following the proof of \eqref{eq:intest.5} we see that we have equality in \eqref{eq:intest.1} with $\varphi = 1/r$ for every point $x\neq y_0 \in \Sigma$. Since all geodesics connecting $y_0$ with other points in $x \in \Sigma$ intersect $\Sigma$ at $x$ non-tangentially, we have that the ambient sectional curvatures
$$\text{sec}_g(\bar{\nabla}r_{y_0} \wedge V) = 0\, ,$$
where $V$ is any unit vector tangent to $S_r:=\partial B_r(y_0) \cap S$ for $0<r<2r_0$. The same argument gives that the principal curvatures along $S$ of $\partial B_r(y_0)$ are equal $1/r$ for $0<r<2r_0$ and that intrinsically $S_r$ is isometric via the exponential map at $y_0$ to $\partial B_r(0) \cap B_{r_0}(r_0 e_3) \subset \R^3$, written in polar coordinates around $0 \in \R^3$. But the Gauss equations then also show that
$$\text{sec}_{\tilde{g}}(V \wedge W) = 0$$
for any two unit vectors $V, W$ tangent to $S_r$ for $0<r<2r_0$. This proves the claim.\\[1ex]
Note that this implies that the mean curvature vector $\mathbf{H}^S(x)$ of $\Sigma \subset S$, seen as a submanifold of $S$ has length $2/r_0$ for all $y \in \Sigma$. Since
$$ \mathbf{H}^S(x) = \pi_{T_xS}\big(\mathbf{H}(x)\big)$$
the choice of $y_0$ in \eqref{eq:intest.13} implies that
$$ \mathbf{H}^S(x) = \mathbf{H}(x) \quad \forall\, x \in \Sigma .$$
It remains to show $S$ is totally geodesic. Pick any point $x_0\in \Sigma$. By the argument before we have that
$$ S =  \exp_{x_0}\big( B_{r_0}(r_0 e_3) \cap \text{span}\{e_1,e_2, e_3\}\big)\, ,$$
where we have chosen $e_1,e_2, e_3$ as before. But this implies that any extrinsic geodesic connecting $x_0$ with $x \neq x_0 \in \Sigma$ has the same length as the intrinsic geodesic in $S$ connecting both points, and thus they both have to concide. This shows that $S$ is totally geodesic, which also implies that $\Sigma$ is totally umbilic in $M$.\\[1ex]
{\bf The equality case for $\kappa >0$:} We can again by scaling assume that $\kappa =1$. The argument is completely analogous to the case $\kappa=0$, the only thing to note is that the equation
$$ 2 = 4 \, \frac{\sinh(|x|)}{2\cosh(|x|) -2}\, \frac{x_3}{|x|} $$
describes the boundary of a geodesic sphere with mean curvature $2$ in normal coordinates around the south pole in the $3$-dimensional model space. 
 \end{proof}

\begin{appendix}

\section{}

\subsection{Strong stationarity}
Let $(M^n,g)$ be a general smooth, complete Riemannian manifold and $S \subset M$ a locally mass minimising rectifiable $m$-current (resp. $m$-dimensional rectifiable flat chain mod 2). The next lemma recalls that $S$ is strongly stationary in the sense of White, compare \cite{EkholmWhiteWienholtz02}.

\begin{lemma}\label{lem:div thm}
Let $S \subset M$ a locally mass minimising rectifiable $m$-current (resp. $m$-dimensional  flat chain mod 2).
There exists an $\ch^{m-1}$-measurable normal vectorfield  $\mathbf{n}$ on $\partial S$ with $\sup |\mathbf{n}| \leq 1$ such that for any vector field $V \in C^1_c (M \times \R)$ it holds
\begin{equation}\label{eq:divthm.1}
 \int \text{{\rm div}}_{S^{\eps_i}}(V)\, d\mu_{S} = \int \langle V, \mathbf{n}\rangle\, d\mu_{\partial S}\ . 
 \end{equation}
\end{lemma}
\begin{proof}
Consider $\phi :\R\times (M\times \R) \ra M\times \R$ such that $\phi(0,x) = x$ and $\tfrac{\partial}{\partial t} \phi = - V \circ \phi$. Since $S^{\eps_i}$ is locally mass minimising we have
$$ \bfM [S^{\eps_i}] \leq \bfM [(\phi(t,\cdot))_\#(S^{\eps_i})] + \bfM[ \phi_\#([0,t]\times \partial S^{\eps_i})]\ .$$
This implies that
$$ \frac{d}{dt}\bigg|_{t=0} \Big(\bfM[(\phi(t,\cdot))_\#(S^{\eps_i})] + \bfM[ \phi_\#([0,t]\times \partial S^{\eps_i})]\Big) \geq 0\, .$$
Using first variation formula and the homotopy formula, see \cite{Simon}, this yields
$$ \int \text{{\rm div}}_{S^{\eps_i}}(V)\, d\mu_{S^{\eps_i}} \leq \int |V^\perp|\, d\mu_{\partial S^{\eps_i}} .$$
The statement follows then from the Riesz representation theorem.
\end{proof}

\subsection{Non-optimal isoperimetric inequality}
We note that one can use the Euclidean isoperimetric inequality for integral currents (resp. flat chains mod $\nu$) to obtain on a Cartan-Hadamard manifold a non-optimal isoperimetric inequality in any dimension and codimension.

\begin{lemma}\label{lem:iso}
Assume $M \in \mathcal{CH}(m+k,0)$ for $k \in \N$ and $K\subset M$ compact. Let $T$ be an integral $m$-current (flat chain mod $\nu$) with $\spt T \subset K$ and $\partial T =0$. Then there exists a constant $C_{K,m} = C(M,K, m)$ and an integral $m+1$-current (flat chain mod $\nu$) $Q$ such that $\partial Q = T$ and
$$ \bfM[Q] \leq C_{K,m}\, \bfM[T]^\frac{m+1}{m}\ .$$
The same holds true on the manifold $M\times \R$ with the standard product metric, provided $\spt T \subset K\times \R$.
\end{lemma}
\begin{proof}
 By picking any basepoint $p \in M$ we can write the metric $g$ of $M$ in exponential coordinates on $T_pM$. Thus on any compact set $K \subset M$ the metric $g$ is uniformly equivalent to the Euclidean metric on $T_pM$. The estimate then follows from the deformation theorem on $\R^{m+k}$ for currents, see for example \cite[Theorem 30.1]{Simon} or respectively for flat chains mod $\nu$, see \cite{White99}.
\end{proof}
\subsection{Avoidance principle in higher codimension} 
We recall White's barrier theorem for mean curvature flow, see \cite[Theorem 14.1]{White_mcfnotes}. We include the proof for completeness.

\begin{theorem}[White]\label{thm:barrier}
 Suppose $\cm$ is the space-time support of an $m$-dimensional integral Brakke flow $(\mu_t)_{t\in I}$ in $\Omega \subset M$. Let $u:\Omega \ra \R$ be a smooth function, so that at $(x_0,t_0)$,
 $$\frac{\partial u}{\partial t} < \text{tr}_m \nabla^2u\, ,$$
 where $\nabla^2u$ is the spacial ambient Hessian, and $\text{tr}_m$ is the sum of the smallest $m$ eigenvalues. Then
 $$ u\big|_{\cm \cap \{t\leq t_0\}}$$
 cannot have a local maximum at $(x_0,t_0)$.
\end{theorem}
\begin{proof} Assume otherwise. We may assume that  $\cm = \cm \cap \{t\leq t_0\}$ and that $u|_\cm$ has a strict local maximum at $(x_0,t_0)$.  (Otherwise we could replace $u$ by $u- (d(x,x_0))^4 - |t_0-t|^2$).

Let $P(r) = B_r(x_0) \times (t_0-r^2,t_0]$. Choose $r>0$ small enough so that $-r^2$ is past the initial time of the flow, $r$ is smaller than the injectivity radius at $x_0$,  $u|_{\cm\cap \overline{P}(r)} $ has a maximum at $(x_0,t_0)$ and nowhere else and $\tfrac{\partial u}{\partial t} < \text{tr}_m \nabla^2u$ on $\bar{P(r)}$. By adding a constant we can furthermore assume that $u_{\cm \cap (\bar{P}\setminus P)}<0< u(x_0,t_0)$. We let $u^+:= \max\{u,0\}$ and plug $(u^+)^4$ into the definition of Brakke flow. Thus
\begin{equation*}
 \begin{split}
 0&\leq \int_{B_r} (u^+)^4 \, d\mu_{t_0} = \int_{B_r} (u^+)^4 \, d\mu_{t_0} - \int_{B_r} (u^+)^4 \, d\mu_{t_0-r^2}\\
 &\leq \int_{t_0-r^2}^{t_0} \int \bigg(\frac{\partial}{\partial t} (u^+)^4 + \langle \mathbf{H}, \nabla(u^+)^4\rangle - |\mathbf{H}|^2 (u^+)^4 \bigg) \, d\mu_t dt\\
 &\leq \int_{t_0-r^2}^{t_0} \int \bigg(\frac{\partial}{\partial t} (u^+)^4 - \text{div}_{\cm}\big(\nabla(u^+)^4\big) \bigg) \, d\mu_t dt\\
 &= \int_{t_0-r^2}^{t_0} \int 4 \bigg((u^+)^3\frac{\partial}{\partial t} u^+-3 (u^+)^2 |\nabla^\cm u^+|^2 - (u^+)^3\text{div}_{\cm}\big(\nabla(u^+)\big) \bigg) \, d\mu_t dt\\
 &\leq  \int_{t_0-r^2}^{t_0} \int 4 (u^+)^3\bigg(\frac{\partial}{\partial t} u^+ - \text{tr}_{m} \nabla^2u^+ \bigg) \, d\mu_t dt < 0\, ,
\end{split}
\end{equation*}
which is a contradiction.
\end{proof}

\subsection{Unique continuation} \label{sec:unique continuation}

For smooth minimal hypersurfaces in a Riemannian manifold, unique continuation follows from the work of Garofalo-Lin \cite{GarofaloLin86, GarofaloLin87}. The case of higher codimension is not treated in there, but follows from work of Kazdan \cite{Kazdan88}, as we now will explain. Assume $\Sigma_1, \Sigma_2$ are smooth, $m$-dimensional immersed minimal surfaces in a smooth Riemannian manifold $(M^{m+k}, g)$ which coincide on a ball $B_\eps(p)$ for some $p \in M$ and some $\eps>0$ sufficiently small. We can assume w.l.o.g.~ that both $\Sigma_1, \Sigma_2$ are embedded in a neighborhood of $\bar{B}_\eps(p)$, otherwise we consider each sheet separately. 

\begin{proposition}
There exists $\delta>0$ such that $\Sigma_1, \Sigma_2$ agree also on $B_{\eps+\delta}(p)$.
\end{proposition}
\begin{proof} Let $B^l_1(0)$ be the unit ball centered at the origin in $\R^l$.  W.l.o.g.~we can work on the set $\Omega:=B^m_1(0) \times B^k_1(0)$ with a metric $h_{ij}$ and the minimal surface $\Sigma$ is given as the graph of a smooth function $u:B^m_1(0) \ra B^k_1(0)$. We denote with $g$ the induced metric on $\Sigma$ and recall the formula
$$ \Delta^{g} v = \text{tr}_{T\Sigma}(\text{Hess}^h(v)) + d(\mathbf{H}) $$
for any ambient function $v:\Omega \ra \R$ where  $\mathbf{H}$ is the mean curvature vector of $\Sigma$. If $\Sigma$ is minimal we obtain the equations
\begin{equation}\label{eq:A.1}
 \Delta^{g} x_l = \text{tr}_{T\Sigma}(\text{Hess}^h(x_l)) 
\end{equation}
 where $x_l$ for $l=1,  \cdots, m+k$ are the standard Euclidean coordinates on $\Omega$. Note that to characterise the minimality of $\Sigma$ it is sufficient to have the above equations fulfilled for $l=m+1, \cdots, m+k$. In the coordinates given by $u = (u_1, \cdots, u_k)$ the right hand side can be written as
 $$f_l(x,u,Du):= - \sum_{i,j =1}^m g^{ij} \left(\bar{\Gamma}_{ij}^l+  2\sum_{r=1}^k \bar{\Gamma}_{m+r\, j}^l \frac{\partial u_r}{ \partial x_i} + \sum_{r,s=1}^k \bar{\Gamma}_{m+r\, m+s}^l \frac{\partial u_r}{\partial x_i} \frac{\partial u_r}{\partial x_j}\right) $$
 where $\bar{\Gamma}_{ij}^k$ are the Cristoffel symbols of $h$, evaluated at the point $(x,u(x))$. Thus the above equations read
 \begin{equation*}
  \frac{\partial}{\partial x_i}\left( \sqrt{\det(g)}g^{ij}\frac{\partial u}{\partial x_j}\right) = \tilde{f}
 \end{equation*}
where $\tilde{f} = (\sqrt{\det(g)} f_{m+1}, \cdots, \sqrt{\det(g)} f_{m+k})$  and
\begin{equation*}%\label{eq:A.3}
  \frac{\partial}{\partial x_i}\left( \sqrt{\det(g)}g^{ij}\right) = \bar{f}_j
 \end{equation*}
for $j =1, \cdots, m$ and $\bar{f}_j =  \sqrt{\det(g)} f_{j}$. Defining the metric $G^{ij} = \sqrt{\det(g)}g^{ij}$ we can rewrite these equations in the form
\begin{align}
 \frac{\partial}{\partial x_i}&\left( G^{ij}\frac{\partial u}{\partial x_j}\right) = \tilde{f} \label{eq:A.2}\\
 \frac{\partial}{\partial x_i}&\left( G^{ij}\right) = \bar{f}_j\qquad \text{for } j =1, \cdots, m\ .\label{eq:A.3}
\end{align}
We writing $G^{ij}(x,u,Du) = G^{ij}[u]$ and similarly $\tilde{f}(x,u,Du) = \tilde{f}[u], \bar{f}_j(x,u,Du) = \bar{f}_j[u]$. Following \cite[\S 8]{SimonWickramasekera16}, we can write the difference of \eqref{eq:A.2} for two solutions $u^1,u^2$ of \eqref{eq:A.1} as
\begin{equation*}
\begin{split}
\tilde{f}[u^1] &- \tilde{f}[u^2] = \frac{\partial}{\partial x_i}\left(G^{ij}[u^1]\frac{\partial u^1}{\partial x_j} - G^{ij}[u^2]\frac{\partial u^2}{\partial x_j} \right) =\\
 &=  
 \frac{\partial}{\partial x_i}\left(\left(G^{ij}[u^1]+G^{ij}[u^2]\right) \frac{\partial}{\partial x_j}\frac{u^1-u^2}{2}  + \left(G^{ij}[u^1] - G^{ij}[u^2]\right)\frac{\partial}{\partial x_j} \frac{u^1+ u^2}{2}\right)\\
 &=  \frac{\partial}{\partial x_i}\left(\bar{G}^{ij} \frac{\partial v}{\partial x_j}  + \left(G^{ij}[u^1] - G^{ij}[u^2]\right)\frac{\partial  \bar{u}}{\partial x_j} \right)\\
 &=  \frac{\partial}{\partial x_i}\left(\bar{G}^{ij} \frac{\partial v}{\partial x_j}\right)  + \left(\tilde{f}_j[u^1] - \tilde{f}_j[u^2]\right)\frac{\partial  \bar{u}}{\partial x_j} +  \left(G^{ij}[u^1] - G^{ij}[u^2]\right)\frac{\partial^2  \bar{u}}{\partial x_i x_j} \ ,
\end{split}
\end{equation*}
where we introduced $v = u^1-u^2, \bar{u}= (u^1-u^2)/2$ and $\bar{G}^{ij} = \left(G^{ij}[u^1] + \bar{G}^{ij}[u^2]\right)/2$ and applied \eqref{eq:A.3}. Assuming that $u^1,u^2$ are bounded in $C^{1,1}$ and using standard interpolation between $u^1$ and $u^2$, this implies that
\begin{equation}\label{eq:A.4}
 \left|  \frac{\partial}{\partial x_i}\left(\bar{G}^{ij} \frac{\partial v}{\partial x_j}\right) \right| \leq C (|v| + |Dv|)
\end{equation}
for some $C \geq 0$. We claim that we can now apply the result of Kazdan, \cite[Theorem 1.8]{Kazdan88} to get the desired result: Note that \cite[(1.9)]{Kazdan88} implies that the operator on the left hand side of \eqref{eq:A.4} is of a form such that \cite[Theorem 1.8]{Kazdan88} is applicable. The estimate \eqref{eq:A.4} implies that  \cite[(1.4)]{Kazdan88} holds with $f(r) =r$ for $m\geq 3$ and $f(r) = r \log(2R_0/r)$ for $m=2$. 

\end{proof}
\end{appendix}

\def\cprime{$'$} \def\weg#1{}
\providecommand{\bysame}{\leavevmode\hbox to3em{\hrulefill}\thinspace}
\providecommand{\MR}{\relax\ifhmode\unskip\space\fi MR }
\providecommand{\MRhref}[2]{%
  \href{http://www.ams.org/mathscinet-getitem?mr=#1}{#2}
}
\providecommand{\href}[2]{#2}

\end{document}